\documentclass{gen-j-l}
\usepackage{amssymb}
\usepackage{amsfonts}

\newtheorem{theorem}{Theorem}[section]
\newtheorem{lemma}[theorem]{Lemma}
\theoremstyle{definition}
\newtheorem{definition}[theorem]{Definition}
\newtheorem{example}[theorem]{Example}

\theoremstyle{remark}
\newtheorem{remark}[theorem]{Remark}
\numberwithin{equation}{section}
\theoremstyle{plain}
\newtheorem{acknowledgement}{Acknowledgement}

\newtheorem{corollary}{Corollary}

\newtheorem{proposition}{Proposition}

\copyrightinfo{2001}{enter name of copyright holder}
\input{tcilatex}

\begin{document}
\title[Restricted ideal theory]{Domains whose ideals meet a universal
restriction}
\author{Muhammad Zafrullah}
\address{Department of Mathematics, Idaho State University, Pocatello, 83209
ID}
\email{mzafrullah@usa.net}
\urladdr{https://www.lohar.com\\
Phone: 2084782759}
\subjclass[2010]{Primary 13A15, 13G05; Secondary 06F05}
\keywords{Lattice, divisorial ideal, $t$-ideal, $t$-invertibility,
Noetherian, Krull, Dedekind, Mori, H-domain}
\dedicatory{Dedicated to my friends}

\begin{abstract}
Let $S(D)$ represent a set of proper nonzero ideals $I(D)$ (resp., $t$%
-ideals $I_{t}(D)$) of an integral domain $D\neq qf(D)$ and let $P$ be a
valid property of ideals of $D.$ We say $S(D)$ meets $P$ (denoted $%
S(D)\vartriangleleft P)$ if each $s\in S(D)$ is contained in an ideal
satisfying $P$. If $S(D)$ $\vartriangleleft P,$ $\dim (D)$ can't be
controlled. When $R=D[X],$ $I(D)$ $\vartriangleleft P$ does not imply $I(R)$ 
$\vartriangleleft P$ while $I_{t}(D)$ $\vartriangleleft P$ implies $I_{t}(R)$
$\vartriangleleft P$ usually. We say $S(D)$ meets $P$ with a twist $($%
written $S(D)\vartriangleleft ^{t}P)$ if each $s\in S(D)$ is such that, for
some $n\in N,$ $s^{n}$ is contained in an ideal satisfying $P$ and study $%
S(D)\vartriangleleft ^{t}P,$ as its predecessor. A modification of the above
approach is used to give generalizations of Almost Bezout domains.
\end{abstract}

\maketitle

\section{\protect\bigskip Introduction\label{S1}}

The general idea of this paper is the following. Consider a property of
ideals in a (commutative) ring $R$ such as "is finitely generated". We raise
and answer questions such as: A commutative ring $R$ is Noetherian if and
only if every ideal of $R$ is finitely generated, what will be a ring every
ideal of which is contained in some finitely generated ideal? It turns out
that this will happen precisely when every maximal ideal of $R$ is finitely
generated. (The resulting ring may not in general be Noetherian.) Let's call
the above process, "tweaking of a property". We note that while the main
thrust of our paper is on tweaking of various properties of ideals of
various kinds in commutative integral domains, the language adopted is such
that it can be used to include questions such as: What will be the result of
tweaking the property, "every left ideal is principal" to "every left ideal
is contained in a principal left ideal. We give examples of domains that
result from this "tweaking". For some examples we will need to use the star
operations called the $v$-operation and the $t$-operation. So it seems best
to start with a brief introduction to those, even before spelling out what
we plan to do.

Let $D$ be an integral domain with quotient field $K\neq D$, let $\mathcal{F}%
(D)$ be the set of nonzero fractional ideals of $D$ and let $f(D)$ be the
set of nonzero finitely generated fractional ideals of $D.$ For $I\in 
\mathcal{F}(D),$ the set $I^{-1}=\{x\in K|xI\subseteq D\}$ is again a
fractional ideal and thus the relation $v$: $I\mapsto I_{v}$ is a function
on $\mathcal{F}(D).$ This function is called the $v$-operation on $D.$
Similarly the relation $t$: $I\mapsto I_{t}=\cup \{F_{v}|$ $0\neq F$ is a
finitely generated subideal of $I\}$ is a function on $\mathcal{F}(D)$ and
is called the $t$-operation on $D.$ These and the operation $d$: $I\mapsto I$
are examples of the so called star operations. The reader may consult
sections 32 and 34 of \cite{Gil 1972} or the first chapter of \cite{Jess
2019} for these operations. However, for the purposes of this introduction,
we note that $I\in \mathcal{F}(D)$ is a $v$-ideal (resp., $t$-ideal) if $%
I=I_{v}$ (resp. $I=I_{t})$ and if $I$ is finitely generated, $I_{v}=I_{t}.$
The rather peculiar definition of the $t$-operation allows one to use Zorn's
Lemma to prove that each integral domain that is not a field has at least
one integral $t$-ideal maximal among integral $t$-ideals, that this maximal $%
t$-ideal is prime and that every proper, integral $t$-ideal is contained in
at least one maximal $t$-ideal. The set of all maximal $t$-ideals of a
domain $D$ is denoted by $t$-$Max(D).$ It can be shown that $D=\cap _{M\in t%
\text{-}Max(D)}D_{M}.$ While we are at it let's also denote by $I(D)$ the
set of all nonzero proper integral ideals of $D$ and by $I_{t}(D)$ the set
of all proper integral $t$-ideals of $D,$ here proper means not equal to $D.$

Now let $S(D)$ represent $I(D)$ ( or $I_{t}(D)$). Let $P$ be a predicate
that defines a non-empty truth set $\Gamma _{S(D)}(P)$ $\subseteq S(D),$
where $P$ can be: "---is invertible" or "---is divisorial",
\textquotedblleft ---is finitely generated" etc.. We say $S(D),$ for a given
value or both values meets $P$ $($written as $S(D)\vartriangleleft P)$ if $%
\forall s\in S(D)\exists \mathcal{\gamma }\in \Gamma _{S(D)}(P)$ $%
(s\subseteq $ $\mathcal{\gamma })$.

(Alternatively, let $P$ be a valid property of ideals and let $\Gamma
_{S(D)}(P)$ be the set of ideals of $D$ satisfying $P,$ we say $S(D)$ meets $%
P$ (denoted as $S(D)\vartriangleleft P$) if each ideal $s$ in $S(D)$ is
contained in some ideal $\mathcal{\gamma }$ that satisfies $P.$ The use of
the predicate or property is to avoid having to state similar theorems over
and over again. A reader can pick up a property of choice to see if it can
be tweaked.)

From an abstract point of view we are actually dealing with a non-empty
poset $(A,\leq )$ such that every member of $A$ precedes at least one
maximal element of $A$. Suppose further that we designate a non-empty subset 
$\Gamma $ of $A$ by some rule. Then every maximal member of $A$ is in $%
\Gamma $ if and only if every member of $A$ precedes some member of $\Gamma
. $ Thus $S(D)\vartriangleleft P\Leftrightarrow Max(D)$ $($resp., $t$-$%
Max(D))\subseteq $ $\Gamma _{S(D)}(P)$. That is easy enough, but the trouble
starts when we ask questions like: Suppose for example $I(D)\vartriangleleft
P$ and suppose $R$ is an extension of $D$ must $I(R)\vartriangleleft P?$
(Same question for $S(D)=I_{t}(D).)$ On the other hand we get the following
benefit from carrying out this study: Take a property $P~$say $"$ finitely
generated", that characterizes commutative Noetherian rings. Then $%
I(D)\vartriangleleft P$ gives us a ring each of whose maximal ideal is
finitely generated. It turns out that this ring is non-Noetherian unless it
is of dimension one. We shall however restrict our attention to integral
domains and note that $D$ is a Krull domain if and only every $t$-ideal of $%
D $ is $t$-invertible. If $P$ stands for "is $t$-invertible" then, as we
shall see, $I_{t}(D)\vartriangleleft P$ is a domain characterized by the
property that every maximal $t$-ideal of $D$ is $t$-invertible. Now you can
set $P$ as: ".. is invertible" and check for yourself that $%
I(D)\vartriangleleft P$ delivers a domain whose maximal ideals are all
invertible but such a domain is not Dedekind unless it is of dimension one.
In fact for each natural number $n$ we can find an $n$ dimensional domain
with each maximal ideal invertible. This fascinating uncontrollability of
Krull dimension is shared by most of $I(D)\vartriangleleft P$ and $%
I_{t}(D)\vartriangleleft P$ etc..

We show in section \ref{Section S2} that if $X$ is an indeterminate over $L$
a field extension of $K,$ and $R=D+XL[[X]],$ and if $P$ returns $T$ on a
maximal ideal $M$ of $D$ if and only if $P$ returns $T$ on $M+XL[[X]],$ $%
S(D) $ $\vartriangleleft P$ if and only if $S(R)\vartriangleleft P$. Also if 
$P$ is such that $P$ returns $T$ on principal ideals, such as "... is
finitely generated", $R=D+XL[X],$ and if $P$ returns $T$ on a maximal ideal $%
M$ of $D$ if and only if $P$ returns $T$ on $M+XL[X],$ $S(D)$ $%
\vartriangleleft P$ if and only if $S(R)\vartriangleleft P.$ Since $\dim
R=\dim D+1,$ \cite[Corollary 1.4]{CMZ 1986}, this shows that if $S(D)$ $%
\vartriangleleft P$ and $P$ returns the truth value $T$ for each principal
ideal, then one can expect no restriction on the Krull dimension of $D$. We
also explain the use of predicate as a way of stating several minor theorems
in one go. Next we show, in section \ref{Section S2}, that if $R=D[X]$ and $%
I(D)$ $\vartriangleleft P,$ then $I(R)$ $\ntriangleleft P$ in cases that we
have considered, yet if $I_{t}(D)$ $\vartriangleleft P$, then $I_{t}(R)$ $%
\vartriangleleft P$ almost always. We give examples to show that generally $%
S(D)\vartriangleleft P$ does not extend to rings of fractions. We study
restrictions, such as requiring the domain to be completely integrally
closed or to be Noetherian etc., that control the dimension of $D$ when $%
S(D)\vartriangleleft P,$ in some cases. In section \ref{Section S3} we study 
$S(D)\vartriangleleft P$ with a twist $($written as $S(D)\vartriangleleft
^{t}P)$ if $\forall s\in S(D)$ $\exists \mathcal{\gamma }\in $ $\Gamma
_{S(D)}(P)(s^{n}$ $\subseteq \mathcal{\gamma }$ for some $n\in N)$ and study 
$S(D)\vartriangleleft ^{t}P$ along the same lines as $S(D)\vartriangleleft P$%
, providing necessary examples. (Here $N$ denotes the set of natural
numbers.) The examples to explain the general idea are mostly well known yet
cover a lot of ground. So the article could be mistaken for a survey. To
allay that notion we have included in section \ref{Section S4} a
generalization of the notion of almost Bezout domains, using a modification
of the principle explained above. Then we study the finite character of this
generalization. However, in view of the ground covered, if someone wishes to
recommend this article as as a survey article they have my permission to
hold this view. Of course our terminology is usually standard, as in \cite%
{Kap 1970} and \cite{Gil 1972}, and we provide adequate introduction to any
term that is new or not quite in common use.)

\section{Effects of a Universal Restriction on $S(D)$ \label{Section S2}}

Even though our main focus will be on the $d$- and $t$-operations, let us
start with an introduction to general star operations so that we can reap
full benefits from our toils. A star operation $\ast $ on $D$ is a function
on $\mathcal{F}(D)$ that satisfies the following properties for every $%
I,J\in \mathcal{F}(D)$ and $0\neq x\in K$:

(i) $(x)^{\ast} = (x)$ and $(xI)^{\ast} = xI^{\ast}$,

(ii) $I \subseteq I^{\ast}$, and $I^{\ast} \subseteq J^{\ast}$ whenever $I
\subseteq J$, and

(iii) $(I^{\ast })^{\ast }=I^{\ast }$.

\vspace{0.1cm} \noindent Now, an ideal $I\in \mathcal{F}(D)$ is a $\ast $%
-ideal if $I^{\ast }=I,$ so a principal ideal is a $\ast $-ideal for every
star operation $\ast .$ Moreover $I\in \mathcal{F}(D)$ is called a $\ast $%
-ideal of finite type if $I=J^{\ast }$ for some $J\in f(D)$. It can be shown
that (a) for every star operation $\ast $ and $I,J\in \mathcal{F}%
(D),~(IJ)^{\ast }=(IJ^{\ast })^{\ast }=(I^{\ast }J^{\ast })^{\ast },$ (the $%
\ast $-multiplication), (b) $(I+J)^{\ast }=(I+J^{\ast })^{\ast }=(I^{\ast
}+J^{\ast })^{\ast }$ (the $\ast $-sum) and (c) $(I^{\ast }\cap J^{\ast
})^{\ast }=I^{\ast }\cap J^{\ast }$($\ast $-intersection).

To each star operation $\ast $ we can associate a star operation $\ast _{s}$
defined by $I^{\ast _{s}}=\bigcup \{\,J^{\ast }\mid J\subseteq I$ and $J\in
f(D)\,\}.$ A star operation $\ast $ is said to be of finite type, or of
finite character, if $I^{\ast }=I^{\ast _{s}}$ for all $I\in \mathcal{F}(D).$
Indeed for each star operation $\ast ,$ $\ast _{s}$ is of finite character.
Thus if $\ast $ is of finite character $I\in \mathcal{F}(D)$ is a $\ast $%
-ideal if and only if for each finitely generated subideal $J$ of $I$ we
have $J^{\ast }\subseteq I.$ Also it is easy to see that $I_{t}=\bigcup
\{\,J_{v}\mid J\subseteq I$ and $J\in f(D)\,\}=I_{v_{s}}$ and so the $t$%
-operation is an example of a star operation of finite character. We will
have occasion to use another star operation called the $w$-operation. It is
defined by $A_{w}=\cap AD_{M}$ where $M\in t$-$Max(D),$ and is finite type.
Star operations of finite character, especially the $t$-operation, will
figure prominently in our discussions. A fractional ideal $I$ is called $%
\ast $-invertible if $(II^{-1})^{\ast }=D.$ It is well known that if $I$ is $%
\ast $-invertible for a finite character star operation $\ast $ then $%
I^{\ast }$ and $I^{-1}$ are of finite type and that every $\ast $-invertible 
$\ast $-ideal is divisorial \cite{Zaf 2000}. If $\ast $ is a star operation
of finite character then just like the $t$-operation, every nonzero proper
integral $\ast $-ideal is contained in a maximal integral $\ast $-ideal that
is prime and just like the $t$-ideals $D=\cap D_{M}$ where $M$ varies over
the maximal $\ast $-ideals of $D.$ We shall be mostly concerned with the two
values of $S(D)$ but will use occasionally $I_{\ast }(D)$ the set of proper,
integral, $\ast $-ideals when we want to go general and not lose sight of
the two values of $S(D).$ (Since $I_{\ast }(D)=I(D)$ $($resp., $I_{t}(D))$
for $\ast =d$ (resp., $\ast =t$)). Let's note that while $I_{\ast }(D)\cup
\{D\}$ is a monoid under the usual $\ast $-multiplication of $\ast $-ideals
with multiplicative identity $D,$ it is a poset under inclusion. From the
poset angle, especially with an eye for p.o group connection $(I_{\ast
}(D)\cup \{D\},+^{\ast },\times ^{\ast }$ $\leq ),$ with $A\leq B$ $%
\Leftrightarrow A\supseteq B,$ is a p.o. monoid and a lattice where $%
A+^{\ast }B=(A,B)^{\ast }=\inf (A,B)=A\wedge B$ and $\sup (A,B)=A\cap B.$ If
we do not have the p.o. group connection in mind, we can consider $(I_{\ast
}(D)\cup \{D\},+^{\ast },\times ^{\ast }$ $\leq ),$ with $A\leq B$ $%
\Leftrightarrow A\subseteq B,$ as a p.o. monoid and a lattice where $%
A+^{\ast }B=(A,B)^{\ast }=\sup (A,B)=A\vee B$ and $\inf (A,B)=A\cap B.$ Both
approaches lead to the same conclusions, though the language undergoes some
changes.  The idea of using a universal restriction via a predicate
germinated in \cite{DZ 2010} where we studied the set $I_{\ast }^{f}(D)$ of
proper $\ast $-ideals of finite type with a preassigned non-empty subset $%
\Gamma $ of $I_{\ast }^{f}(D),$ requiring that every pair of members with $%
A+^{\ast }B\in $ $I_{\ast }^{f}(D),$ $A,B$ be contained in some member of $%
\Gamma .$ (This is equivalent to saying that every proper ideal in $I_{\ast
}^{f}(D)$ is contained in a member of $\Gamma ,$ hence the current
approach.) As these studies appeal mostly to partial order, they stand to
have applications in other areas, as well.

We start with a simple example to set the scene. Let's consider, for a star
operation $\ast $ of finite character, $I_{\ast }(D)$ and define $\Gamma
_{I_{\ast }(D)}(P)$ with $P=$ "---is principal" and suppose that $I_{\ast
}(D)\vartriangleleft P.$ Then every maximal $\ast $-ideal of $D$ is
principal, as we have already observed. But the story doesn't end here. The
event of $I_{\ast }(D)\vartriangleleft P$ imparts some properties to $D,$
such as: the only atoms (irreducible elements) in $D$ are primes and hence
generators of maximal $\ast $-ideals. For this note that if $a$ is an atom
then $a$ must belong to a maximal $\ast $-ideal, which is principal and
hence generated by a prime $p$. But then $a$ is an associate of $p$ and
hence a prime. Thus an irreducible element is a prime in $D,$ if $I_{\ast
}(D)\vartriangleleft P$ for any star operation $\ast $ of finite character.
Now for $\ast =d$ the identity operation $I(D)\vartriangleleft P$ gives a
domain $D$ in which every proper nonzero ideal is contained in a principal
ideal, something stronger than what Cohn \cite{Coh 1968} called a pre-Bezout
domain (every pair of coprime elements is co-maximal). In fact $%
I(D)\vartriangleleft P$ gives a domain something that is even stronger than
what was called a special pre-Bezout, or spre-Bezout domain in \cite{DZ 2010}%
. (Recall that $D$ is a spre-Bezout domain if every finite co-prime set of
elements is co-maximal.) Similarly if $I_{t}(D)\vartriangleleft P$, then $D$
is something stronger than a PSP-domain (every primitive polynomial over $D$
is super-primitive), also discussed in \cite{DZ 2010}. Recall that a
polynomial $f$ is super primitive if $(A_{f})_{v}=D$, where $A_{f}$ is the
content, the ideal generated by the coefficients, of $f.$ Now it is easy to
see that if such a domain is atomic, it is at least a UFD (when $%
I_{t}(D)\vartriangleleft P)$ and a PID (when $I(D)\vartriangleleft P).$ Now,
can we find domains that satisfy these properties and yet are not atomic?
Yes indeed!

\begin{example}
\label{Example TX0} (1) Let $Z,Q$ denote the ring of integers and its
quotient field respectively and let $X$ be an indeterminate over $Q,$ then
the ring $D=Z+XQ[X]$ is such that $I(D)\vartriangleleft P$, where $P=$
"---is principal". (2) Let $Z,L$ denote the ring of integers and a field
containing its quotient field respectively and let $X$ be an indeterminate
over $L,$ then the ring $R=Z+XL[[X]]$ is such that $I(R)\vartriangleleft P$,
where $P=$ "---is principal".
\end{example}

Illustration: According to \cite[Theorem 4.21]{CMZ 1978} the nonzero prime
ideals of $D$ are of the form $pZ+XQ[X],XQ[X]$ and maximal height one
principal primes of the form $f(X)D$ where $f(X)$ is irreducible in $Q[X]$
and $f(0)=1.$ Now $XQ[X]$ is not maximal and the rest of them are. So all
the maximal ideals are principal and so $I(D)\vartriangleleft P$ with $P$
given above. That $D$ is not atomic can be concluded from the fact that $X$
cannot be expressed as a product of atoms. For (2) it is easy to check that
every maximal ideal of $R$ is principal, as the maximal ideals of $Z+XL[[X]]$
are of the form $pZ+XL[[X]]$ where $p$ is a prime element of $Z.$

Now according to \cite{CMZ 1978}, $\dim D=2$ and we said that if $%
I(D)\vartriangleleft P,$ then there maybe no restriction on $\dim D.$ The
answer to this question is provided in a more general form below.

Let's first collect some simple results, observations and notation. We say
that $P$ returns $T$ on an ideal of $I(D)$ if the truth value of $P$ for
that ideal is $T.$ For the sake of easy reference, let's start with an
observation that we have already made.

\begin{lemma}
\textbf{\ }\label{Lemma TX1}Let $(A,\leq )$ be a non-empty poset such that
every element of $A$ precedes some maximal element of $A$ and suppose that
we can designate a non-empty subset $\Gamma $ of $A$ by some rule. Also let $%
Max(A)$ denote the set of all maximal elements of $A.$ Then every member of $%
A$ precedes some member of $\Gamma $ if and only if $Max(A)\subseteq \Gamma
. $ Thus $I(D)\vartriangleleft P$ if and only if $P$ returns $T$ for each
member of $Max(D)$ and $I_{t}(D)\vartriangleleft P$ if and only if $P$
returns $T$ for each member of $t$-$Max(D).$
\end{lemma}

This, somewhat simple observation may, in some instances, have some
interesting consequences.

\begin{lemma}
\label{Lemma TX2}(1) If a maximal ideal $M$ of $D$ is a $t$-invertible $t$%
-ideal, then $M$ is invertible. (2) If $P_{1}=$ "---is $t$-invertible" and $%
P_{2}=$ "--- is invertible", then $I(D)\vartriangleleft P_{1}\Leftrightarrow
I(D)\vartriangleleft P_{2}$ and (3) $I(D)\vartriangleleft P$ $\Rightarrow
I_{t}(D)\vartriangleleft P$ for any predicate $P$ whose truth set consists
of $t$-ideals.
\end{lemma}

\begin{proof}
(1) (This is well known, but we include the proof for completeness.) Suppose 
$M$ is a $t$-invertible $t$-ideal then $(MM^{-1})_{t}=D.$ If $MM^{-1}\neq D$
then $MM^{-1}$ must be contained in a maximal ideal $N.$ But since $%
M\subseteq MM^{-1},N=M.$ So $MM^{-1}\subseteq M.$ But as $M$ is also a $t$%
-ideal, $D=(MM^{-1})_{t}\subseteq M,$ a contradiction.

(2) By Lemma \ref{Lemma TX1}, $I(D)\vartriangleleft P_{i}$ $\Leftrightarrow
P_{i}$ returns $T$ for each maximal ideal $M$ and for each $i=1,2.$ So $%
I(D)\vartriangleleft P_{1}\Rightarrow $ every maximal ideal is a $t$%
-invertible $t$-ideal and by (1) every maximal ideal is invertible. So $%
I(D)\vartriangleleft P_{1}\Rightarrow I(D)\vartriangleleft P_{2}.$ The
converse is obvious because every invertible ideal is a $t$-invertible $t$%
-ideal.

(3) Suppose that $I(D)\vartriangleleft P$ then, in particular, for every
maximal $t$-ideal $M,$ $P$ returns $T$.
\end{proof}

\begin{proposition}
\label{Proposition UX0}(1) Let, on $I(D),$ $P=$ "--- is a principal ideal
(resp., $t$-invertible $t$-ideal, $t$-ideal of finite type, $t$-ideal,
finitely generated ideal,divisorial ideal). Then $I(D)\vartriangleleft P$ if
and only if every maximal ideal of $D$ is a principal ideal (resp.,
invertible ideal, $t$-ideal of finite type, $t$-ideal, finitely generated
ideal, divisorial ideal) of $D.$ (2) Let, on $I(D),$ $P=$ "--- is a
principal ideal (resp., invertible ideal, $t$-invertible $t$-ideal, $t$%
-ideal of finite type, finitely generated ideal, divisorial ideal). Then $%
I_{t}(D)\vartriangleleft P\Leftrightarrow $ every maximal $t$-ideal is a
principal ideal (resp., invertible ideal, $t$-invertible $t$-ideal, $t$%
-ideal of finite type, finitely generated ideal, divisorial ideal).
\end{proposition}

\begin{proof}
In the presence of Lemma \ref{Lemma TX1} and Lemma \ref{Lemma TX2}, it
appears totally unnecessary to repeat the arguments required for the proofs
of (1) and (2).
\end{proof}

Note that in case of (1) every maximal ideal being a $t$-ideal of finite
type ensures that every maximal $t$-ideal of $D$ is actually a maximal
ideal. Indeed if we suppose that $\wp $ is a maximal $t$-ideal that is not
maximal, then $\wp $ is contained in a maximal ideal, say $M,$ but $M$ is
already a $t$-ideal.

We have restricted our attention to the star operations that are easily
defined for usual extensions. One of the usual extensions is the $D+XL[X]$
construction, where $L$ is an extension of $K$ and $X$ is an indeterminate
over $L.$ It is a special case of the $D+M$ construction of \cite{BR 1976}.
To be able to fully appreciate how it works, one needs to learn a little
about the construction $D+XL[X].$ Let $D,L,X$ be as above$.$ Then $R=D+XL[X]$
$=\{f\in L[X]|f(0)\in D\}$ is an integral domain. Indeed $R$ has two kinds
of nonzero prime ideals $P$ , ones that intersect $D$ trivially and ones
that don't. If $P\cap D\neq (0)$ then $P=P\cap D+XL[X]$ \cite[Lemma 1.1]{CMZ
1986} and obviously $P$ is maximal if and only if $P\cap D$ is. It can be
shown, as was indicated prior to the proof of Corollary 16 in \cite{ACZ 2015}%
, that if $P=P\cap D+XL[X]$, then $P$ is a maximal $t$-ideal of $R$ if and
only if $P\cap D$ is a maximal $t$-ideal of $D$ and indeed as $P_{v}=(P\cap
D)_{v}+XL[X],$ $P$ is divisorial if and only if $(P\cap D)$ is. Moreover,
prime ideals of $R$ that are not comparable with $XL[X],$ i.e. ones that
intersect $D$ trivially, are of the form $(1+Xg(X))R$ where $1+Xg(X)$ is an
irreducible element of $L[X],$ \cite[Lemmas 1.2, 1.5]{CMZ 1986}. (This can
also be seen as follows: If $P$ is a prime that intersects $D$ trivially,
then $P$ extends to a prime $\wp $ of $K+XL[X]$ that is incomparable with $%
XL[X].$ Now $K+XL[X]$ is one dimensional and every element of $K+XL[X]$ is
of the form $lX^{r}(1+Xg(X))$ where. $l\in L,$ $r\geq 0$ and $1+Xg(X)$ is
obviously a product of primes from $L[X]$. Next $lX^{r}(1+Xg(X))\in \wp $
forces $(1+Xg(X))\in \wp ,$ because $X\notin \wp .$ But then $\wp $ is
principal generated by a prime of the form $1+h(X)$ and this also is a prime
in $R,$ thus $1+Xh(X)\in \wp \cap R=P.$ Now as $P$ contains a principal
prime $(1+Xh(X))R$ that extends to a maximal height one prime in $K+XL[X],$ $%
P=(1+Xh(X))R.)$ Also as $XL[X]$ is of height one $XL[X]$ is a $t$-ideal and $%
\dim R=\dim D+1,$ by \cite[Corollary 1.4]{CMZ 1986}. Let us say that a
predicate $P$ respects principals if $P$ returns $T$ on principal ideals
(i.e. principal ideals satisfy $P)$ .

\begin{theorem}
\label{Theorem UX0A}A. Let $P$ be a predicate that respects principals$,$ $L$
an extension field of $K,$ $X$ an indeterminate over $L$ and let $R=D+XL[X]$%
. Then (i) given that $P$ returns $T$ on a maximal ideal $M$ of $D$ if and
only if $P$ returns $T$ on $M+XL[X],$ $I(D)\vartriangleleft P$ $%
\Leftrightarrow I(R)\vartriangleleft P$ (ii) given that $P$ returns $T$ on a
maximal $t$-ideal $M$ of $D$ if and only if $P$ returns $T$ on $M+XL[X],$ $%
I_{t}(D)\vartriangleleft P\Leftrightarrow I_{t}(R)\vartriangleleft P.$ B.
Let $P$ be a predicate$,$ $L$ an extension field of $K,$ $X$ an
indeterminate over $L$ and let $R=D+XL[[X]]$. Then (iii) given that $P$
returns $T$ on a maximal ideal $M$ of $D$ if and only if $P$ returns $T$ on $%
M+XL[[X]],$ $I(D)\vartriangleleft P$ $\Leftrightarrow I(R)\vartriangleleft P$
and (iv) given that $P$ returns $T$ on a maximal $t$-ideal $M$ of $D$ if and
only if $P$ returns $T$ on $M+XL[[X]],$ $I_{t}(D)\vartriangleleft
P\Leftrightarrow I_{t}(R)\vartriangleleft P.$
\end{theorem}

\begin{proof}
(Perhaps, before a "formal" proof of (i), an example might help. Take a
predicate $P$, say: \_\_ is finitely generated. Then $P$ respects principals
because a principal ideal is finitely generated. Now given a maximal ideal $%
M $ of $D$ the ideal $M+XL[X]=MR$ \cite[Lemma 1.1 and Theorem 1.3]{CMZ 1986}
is a maximal ideal of $R$ and obviously $M$ is finitely generated if and
only if $MR$ is. Then $I(D)\vartriangleleft P$ implies $I(R)\vartriangleleft
P$ because every maximal ideal of $D$ being finitely generated implies every
maximal ideal of $R$ of the form $M+XL[X]$ being finitely generated and as
all the maximal ideals that intersect $D$ trivially are principal, we
conclude that every maximal ideal of $R$ is finitely generated i.e. $%
I(R)\vartriangleleft P.$ Conversely suppose $I(R)\vartriangleleft P.$ That
is every maximal ideal of $R$ is finitely generated. Then in particular
maximal ideals of $R$ that intersect $D$ non trivially are finitely
generated. But the ideals of $R$ that intersect $D$ non trivially are of the
form $M+XL[X]=MR$ \cite[Lemma 1.1 and Theorem 1.3]{CMZ 1986}. Now each $MR$
being finitely generated implies that each maximal ideal $M$ of $D$ is
finitely generated. But this means $I(R)\vartriangleleft P\Rightarrow
I(D)\vartriangleleft P.)$

(i) Suppose $I(D)\vartriangleleft P,$ then $P$ returns $T$ for every maximal
ideal $M$ of $D$ and hence for every maximal ideal of $R$ of the form $%
M+XL[X],$ by the given. Since $P$ respects principal ideals we conclude that 
$P$ returns $T$ for every maximal ideal of $R.$ (Since every maximal ideal
of $R$ not of the form $M+XL[X]$ is principal.) That is $I(R)%
\vartriangleleft P.$ Conversely suppose that $I(R)\vartriangleleft P.$ Then $%
P$ returns $T$ for all maximal ideals $\mathcal{M}$ of $R,$ in particular
for the ones that intersect $D$ non-trivially. But those are precisely of
the form $\mathcal{M}=\mathbf{m}+XL[X]$ where $\mathbf{m}=\mathcal{M}\cap D$
is maximal and as $P$ returns $T$ for $\mathbf{m}+XL[X]$ if and only if $P$
returns $T$ for $\mathbf{m},$ and as the $\mathbf{m}s$ are precisely the
maximal ideals of $D$ we conclude that $I(D)\vartriangleleft P.$ The proof
of (ii) follows the same lines as those adopted in the proof of (i).
However, just for completeness we include it. Suppose $I_{t}(D)%
\vartriangleleft P$ then $P$ returns $T$ for every maximal $t$-ideal $M$ of $%
D$ and hence for every maximal $t$-ideal of $R$ of the form $M+XL[X].$ Since 
$P$ respects principal ideals we conclude that $P$ returns $T$ for every
maximal $t$-ideal of $R.$ That is $I_{t}(R)\vartriangleleft P.$ Conversely
suppose that $I_{t}(R)\vartriangleleft P.$ Then $P$ returns $T$ for all
maximal $t$-ideals $\mathcal{M}$ of $R,$ in particular for the ones that
intersect $D$ non-trivially. But those are precisely of the form $\mathcal{M}%
=\mathbf{m}+XL[X]$ where $\mathbf{m}=\mathcal{M}\cap D$ is a maximal $t$%
-ideal and as $P$ returns $T$ for $\mathbf{m}+XL[X]$ if and only if $P$
returns $T$ for $\mathbf{m},$ and as the $\mathbf{m}s$ are precisely the
maximal $t$-ideals of $D$ we conclude that $I_{t}(D)\vartriangleleft P.$ For
(iii) and (iv) all one has to note is that $m$ is a maximal ($t$-) ideal of $%
D$ if and only if $m+XL[[X]]$ is and that $M$ is a maximal ($t$-) ideal of $%
D+XL[[X]]$ if and only if $M=m+XL[[X]]$ is where $m$ is a maximal ($t$-)
ideal of $D.$
\end{proof}

The above "theorem" is not much of a theorem, really. But it tells us what
to check for, before making a statement such as $I(D)\vartriangleleft P$ $%
\Leftrightarrow I(R)\vartriangleleft P.$

\begin{corollary}
\label{Corollary UX0B}(i) With $D,L,X,R$ as in (i) Theorem \ref{Theorem UX0A}
and with $P=$ "--- is a principal ideal (resp., invertible ideal, $t$%
-invertible $t$-ideal, $t$-ideal of finite type, $t$-ideal, finitely
generated ideal, divisorial ideal) $I(D)\vartriangleleft P\Leftrightarrow
I(R)\vartriangleleft P$ and (ii) with $D,L,X,R$ as in (ii) of Theorem \ref%
{Theorem UX0A} and with $P=$ "--- is a principal ideal (resp., invertible
ideal, $t$-invertible $t$-ideal, $t$-ideal of finite type, finitely
generated ideal, divisorial ideal) $I_{t}(D)\vartriangleleft
P\Leftrightarrow I_{t}(R)\vartriangleleft P.$ (iii) with $D,L,X,R$ as in
(iii) Theorem \ref{Theorem UX0A} and with $P=$ "--- is a principal ideal
(resp., $t$-invertible $t$-ideal, $t$-ideal of finite type, $t$-ideal, $t$%
-ideal, finitely generated ideal, divisorial ideal) $I(D)\vartriangleleft
P\Leftrightarrow I(R)\vartriangleleft P$ and (iv) with $D,L,X,R$ as in (iv)
of Theorem \ref{Theorem UX0A} and with $P=$ "--- is a principal ideal
(resp., $t$-invertible $t$-ideal, $t$-ideal of finite type, $t$-ideal,
finitely generated ideal, divisorial ideal) $I_{t}(D)\vartriangleleft
P\Leftrightarrow I_{t}(R)\vartriangleleft P.$
\end{corollary}

\begin{proof}
(i)Note that in each case $P$ returns $T$ for a principal ideal. Moreover
for $A$ an ideal of $D,$ because $A_{v}+XL[X]=(A+XL[X])_{v}$ and $%
A_{t}+XL[X]=(A+XL[X])_{t}$ and because $A+XL[X]=A(D+XL[X]),$ $A$ being
finitely generated, invertible (or being a $v$-ideal of finite type) results
in $A+XL[X]$ being of that kind and vice versa, we conclude that the
requirements of Theorem \ref{Theorem UX0A} are met. (Indeed as a maximal
ideal being a $t$-invertible $t$-ideal is invertible, we haven't let
anything unverified.) For (ii) note that all the checking is as in (i), even
the $t$-invertible $t$-ideal case falls under $t$-ideals of finite type and $%
t$-ideals of finite type are all $v$-ideals. So nothing more needs be done.
For (iii) and (iv) note that all except one statements for $P$ follow
through.
\end{proof}

\begin{remark}
\label{Remaek UX0C}Note that if $D$ is not a field, as we have assumed from
the start, then, whatever be $D,$ $D+XL[X]$ is not Noetherian. This is
because $D+XL[X]$ affords a strictly ascending chain of ideals such as $%
(X)\subseteq (X/d)\subseteq (X/d^{2})\subseteq $ $...\subseteq (X/d^{n})$
for any nonzero non unit $d$ of $D.$ Now as the maximal ideals of a
Noetherian domain $D$ are finitely generated so are the maximal ideals of $%
D+XL[X]$, by Corollary \ref{Corollary UX0B}. This gives us an example (a) of
a non-Noetherian domain whose maximal ideals are all finitely generated.
That is not all, we can construct chains of such domains, of any length,
starting with a domain whose maximal ideals are all finitely generated. To
make things simple let $L=K.$ Let $R_{0}$ be a domain with the property that
every maximal ideal of $R_{0}$ is finitely generated and let $%
R_{1}=R_{0}+X_{0}qf(R_{0})[X_{0}],$ where $X_{0}$ is an indeterminate over $%
qf(R_{0}),$ $R_{2}=R_{1}+X_{1}qf(R_{1})[X_{1}],$ where $X_{1}$ is an
indeterminate over $qf(R_{1})$ and obviously every maximal ideal of $R_{2}$
is finitely generated because $R_{1}$ has this property. If proceeding in
this manner, we reach $R_{n}=R_{n-1}+X_{n-1}qf(R_{n-1})[X_{n-1}]$, where $%
X_{n-1}$ is an indeterminate over $qf(R_{n-1})$ we can construct the next.
As a result of this recursive procedure we have a chain of domains: $%
R_{0}\subseteq R_{1}\subseteq ...$ $\subseteq R_{n}\subseteq
R_{n+1}\subseteq ...,$ where each of $R_{i}$ gets the property of having all
maximal ideals finitely generated from the previous, for $i>0$. Next recall
that (b) $D$ is a Mori domain if $D$ has ACC on integral divisorial ideals.
Obviously Noetherian domains and less obviously Krull domains are Mori. It
can be shown that $D$ is a Mori domain if and only if for every nonzero
integral ideal $A$ of $D$ there is a finitely generated ideal $F\subseteq A$
such that $A_{v}=F_{v}$ \cite[Lemma 1]{Nishi 1963}. This translates to:
every $t$-ideal is a $t$-ideal of finite type \cite[Corollary 1.2]{AA 1988}.
Thus if $D$ is Mori, then every maximal $t$-ideal of $D$ is of finite type.
To show that the property of having every maximal $t$-ideal of finite type
does not characterize Mori domains one can construct $R=D+XK[X]$ indicating,
via Corollary \ref{Corollary UX0B}, that every maximal $t$-ideal of $R$ is
of finite type but $R$ is not Mori because $R$ affords an ascending chain
like: $(X)\subseteq (X/d)\subseteq (X/d^{2})\subseteq $ $...\subseteq
(X/d^{n})$ for any nonzero non unit $d$ of $D.$ We can actually construct,
as in (a) above, chains of domains satisfying this property. We can, of
course do the above with $D+XL[[X]]$ or with any variation of it, i.e.
Gilmer's $D+M$ construction \cite{BG 1973}.
\end{remark}

There are other uses Corollary \ref{Corollary UX0B} can be put to, but we
shall let the reader discover those, if need arises. We now concentrate on
the next extension $R=D[X]$ where $X$ is the usual indeterminate over $D.$

\begin{proposition}
\label{Proposition UXA} (1) Let $I(D)\vartriangleleft P$ where $P=$ "--- is
a proper nonzero principal ideal (resp., $t$-invertible $t$-ideal, $t$%
-ideal, $t$-ideal of finite type, divisorial ideal), let $X$ be an
indeterminate over $D$ and let $R=D[X].$ Then it never is the case that $%
I(R)\vartriangleleft P$ for $P=$ "--- is a proper nonzero principal ideal
(resp., $t$-invertible $t$-ideal, $t$-ideal, $t$-ideal of finite type,
divisorial ideal ) and (2) Let $I_{t}(D)\vartriangleleft P$ where $P=$ "---
is a $t$-invertible $t$-ideal (resp., $t$-ideal of finite type, divisorial
ideal), let $X$ be an indeterminate over $D$ and let $R=D[X].$ Then $%
I_{t}(R)\vartriangleleft P$ where $P=$ "--- is a $t$-invertible $t$-ideal
(resp., $t$-ideal of finite type, divisorial ideal) and conversely.
\end{proposition}

\begin{proof}
(1) (This can be considered well known, but is included for completeness.)
Let $I(D)\vartriangleleft P$ where $P=$ "--- is a proper nonzero principal
ideal (resp., $t$-invertible $t$-ideal, $t$-ideal, $t$-ideal of finite type,
divisorial ideal). Then every maximal ideal $\wp $ of $D$ is a $t$-ideal.
Now consider the prime ideal $\wp \lbrack X]$ in $R=D[X]$ and note that $\wp
\lbrack X]$ can never be a maximal ideal because $D[X]/\wp \lbrack X]\cong
(D/\wp )[X]$ is a polynomial ring over a field and so must have an infinite
number of maximal ideals. This forces $\wp \lbrack X]$ to be properly
contained in an infinite number of maximal ideals $M_{\alpha }$ of $D[X].$
Let $M$ be one of them$.$ Then $M=(f,\wp \lbrack X]).$ Now, if it were the
case that $I(R)\vartriangleleft P$ for $P=$ "--- is a proper $t$-ideal",
then every maximal ideal of $R$ would be a $t$-ideal. This would make $M$ a $%
t$-ideal with $M\cap D=\wp \neq (0).$ But then, according to Proposition 1.1
of \cite{HZ 1989}, $M=(M\cap D)[X]=\wp \lbrack X],$ a contradiction to the
fact that $\wp \lbrack X]\subsetneq M$. For (2) note that if $%
I_{t}(D)\vartriangleleft P$ where $P$ is as specified, then every maximal $t$%
-ideal $\wp $ of $D$ is a $t$-invertible $t$-ideal (resp., $t$-ideal of
finite type, divisorial ideal). Now let $M$ be a maximal $t$-ideal of $R.$
If $M\cap D=(0)$, then $M$ is a $t$-invertible $t$-ideal and hence a $t$%
-ideal (and divisorial, being a finite type $t$-ideal), by Theorem 1.4 of 
\cite{HZ 1989}. Next if $M$ is such that $M\cap D\neq (0),$ then $M=(M\cap
D)[X]$ where $M\cap D$ is a maximal $t$-ideal of $D$ and hence a $t$-ideal,
and obviously is divisorial if and only if $M$ is divisorial \cite[%
Proposition 4.3]{HH 1980}. Conversely suppose that $I_{t}(R)\vartriangleleft
P$ for $P$ as specified. Then every maximal $t$-ideal $M$ of $R$ is a $t$%
-invertible $t$-ideal (resp., $t$-ideal of finite type, divisorial ideal).
Now let $\wp $ be a maximal $t$-ideal of $D.$ Then $\wp \lbrack X]$ is a
maximal $t$-ideal of $R$ by Proposition 1.1 of \cite{HZ 1989} and hence
divisorial. But this leads to $\wp \lbrack X]=(\wp \lbrack X])_{v}=\wp
_{v}[X]$ (resp., $(\wp \lbrack X])_{t}=\wp _{t}[X])$ and hence to $\wp =\wp
_{v}.$ (We have chosen to focus on divisorial ideals ($t$-ideals), as all
the other cases are divisorial (or $t$-ideals) and a maximal $t$-ideal of $R$
that intersects $D$ trivially is divisorial of finite type and hence a $t$%
-ideal.) Moreover if a maximal $t$-ideal $M$ of $R$ intersects $D$
non-trivially then $M=(M\cap D)[X]$ as above and of course $M$ is a $t$%
-ideal ( $t$-ideal of finite type, divisorial) if and only if $M\cap D$ is) 
\cite[Proposition 4.3]{HH 1980}.
\end{proof}

I cannot find a way of proving or disproving the following: Let $R=D[X],$
and let $P=$ "--- is a finitely generated ideal" then $I(D)\vartriangleleft
P\nRightarrow I(R)\vartriangleleft P.$

Now we are ready to show that if $R=D_{\mathcal{S}},$ for a multiplicative
set $\mathcal{S}$ of $D$ where $S(D)\vartriangleleft P$ for $P=$ "--- is a
proper nonzero principal ideal (resp., $t$-invertible $t$-ideal, $t$-ideal, $%
t$-ideal of finite type, divisorial ideal), then it may not generally be the
case that $S(R)\vartriangleleft P$. Let's first recall from Lemma \ref{Lemma
TX2} that if a maximal ideal is a $t$-invertible $t$-ideal then it is
actually invertible. Before we start constructing examples, let's take a
look at the tool that we use in the following example. Let $K$ be a proper
subfield of a field $L$, let $X$ be an indeterminate over $L$ and let $%
T=K+XL[X]$. The ring $T$ is an example of an atomic domain that is not a UFD
(see \cite[page 353]{Coh 1989}) and an example of a $D+M$ construction. That 
$T$ is one dimensional follows from \cite[Corollary 1.4]{CMZ 1986}, that
every maximal ideal of $T$ different from $XL[X]$ is principal of height one
follows from Lemmas 1.2 and 1.5 of \cite{CMZ 1986} and that $XL[X]$ is
divisorial can be easily checked, because $XL[X]=(X,lX)_{v}$ where $l\in
L\backslash K.$

\begin{example}
\label{Example UXD}Let $L$ be a field extension of $K$ with $[L:K]=\infty ,$
let $X$ be an indeterminate over $L$ and consider $R=D+XL[X].$ Set $%
S=D\backslash (0).$ If every maximal ideal of $D$ is principal (invertible,
finitely generated) then so is every maximal ideal of $R.$ But that is not
the case for every maximal ideal of $R_{S}.$ For $R_{S}=K+XL[X]$ has a
maximal ideal that is a $t$-ideal but neither principal nor finitely
generated, because $[L:K]=\infty .$ (It is easy to see that every invertible
ideal is principal in $T$, \cite[Example 1.10]{BZ 1988}.)
\end{example}

The following example has been taken, almost verbatim, from \cite[Example 3.3%
]{HZ 2015}. To decipher this example, recall that $D$ is a PVMD (Prufer $v$%
-multiplication domain) if every nonzero finitely generated ideal of $D$ is $%
t$-invertible. A good source for this concept is \cite{MZ 1981}.

\begin{example}
\label{Example UXE}. There does exist at least one example of a domain $D$
such that each maximal ideal of $D$ is a $t$-ideal but for some maximal
ideal $M$ we have $MD_{M}$ not a $t$-ideal. One such example is that of an
essential domain that is not a PVMD. (Recall that an integral domain $D$ is
essential if $D$ has a set $G$ of primes such that $D_{p}$ is a valuation
domain for each $P\in G$ and $D=\cap _{P\in G}D_{P}.)$ Now the example in
question was constructed by Heinzer and Ohm in \cite{HO 1973} and further
analyzed in \cite{MZ 1981} and \cite{GHL 2004}. As it stands, the example
has all except one maximal ideals of height one and hence $t$-ideals and the
other maximal ideal $M$ is a height $2$ prime $t$-ideal. Indeed this is the
maximal ideal $M$ such that $D_{M}$ is a $2$-dimensional regular local ring
and so with a maximal ideal that is not a $t$-ideal. Showing that while $%
I(D)\vartriangleleft P$ for $P=$ "--- is a $t$-ideal of $D$, $%
I(D_{M})\ntriangleleft P$.
\end{example}

For the next example recall from \cite{Zaf 1987} that an integral domain $D$
is a pre-Schreier domain if for all $a,b_{1},b_{2}\in D\backslash \{0\},$ $%
a|b_{1}b_{2}$ implies that $a=a_{1}a_{2},$ with $a_{i}\in D$ such that $%
a_{i}|b_{i}.$ Also call a $D$-module $M$ locally cyclic if for any elements $%
x_{1},x_{2},...,x_{n}\in M$ there is a $d\in M$ such that $x_{i}=r_{i}d$ for
some $r_{i}\in D.$

\begin{example}
\label{Example UXF}For $%
\mathbb{R}
$ the field of real numbers, let $%
\mathbb{R}
+M$ , be a non-discrete rank one valuation domain, as constructed in say
Example 4.5 of \cite{Zaf 1987}. As decided in the above-mentioned example, $%
\,T=%
\mathbb{Q}
+M$ (where $%
\mathbb{Q}
$ is the field of rational numbers) is a pre-Schreier domain with $M$
divisorial and by \cite[Theorem 4.4 ]{Zaf 1987} locally cyclic. But then $M$
cannot be a $v$-ideal of finite type. For if $M=(x_{1},x_{2},...,x_{n})_{v},$
then there would be a $d\in M$ such that $M=(x_{1},x_{2},...,x_{n})_{v}%
\subseteq (d)\subseteq M,$ contradicting the construction in Example 4.5 of 
\cite{Zaf 1987}. Now let $p$ be a prime element in $%
\mathbb{Z}
$, the ring of integers, and consider the local ring $R=%
\mathbb{Z}
_{(p)}+M.$ Indeed the maximal ideal of $R$ is principal and hence can pass
as a $t$-ideal of finite type, a $t$-invertible $t$-ideal. But if $S$ is the
multiplicative set of $R$ generated by $p,$ neither of these properties are
shared by the maximal ($t$-) ideal $M$ of $R_{S}=%
\mathbb{Q}
+M.$
\end{example}

Now the fact that $I(D)\vartriangleleft P$ can go through the $D+XL[X]$
construction with the various descriptions of $P$ can be used to construct,
for example, a domain of any (finite) dimension with $t$-maximal ideals
principal. If that reminds an attentive reader of comments (3) and (4) of
Remarks 8 of \cite{MZ 1990}, then so be it. The point however is that the
events of $I(D)\vartriangleleft P$ and $I_{t}(D)\vartriangleleft P,$ with
suitable descriptions of $P$, do not have the usual Ascending Chain
Condition (ACC) on ideals (principal or $t$-) ideals. One may wonder if
there are any simple restrictions that will get the beast under control. Yet
to prepare to see that, here is another simple set of results that can come
in handy when we are dealing with completely integrally closed integral
domains. Of course before we bring in those results some introduction is in
order. Recall that an integral domain $D$ with quotient field $K$ is
completely integrally closed if whenever $rx^{n}$ $\in D$ for $x\in K$, $%
0\neq $ $r\in D$, and every integer $n\geq 1$, we have $x\in D$. It can be
shown that an intersection of completely integrally closed domains is
completely integrally closed. The go to reference for Krull domains is
Fossum's book \cite{Fos 1973} where you can find that $D$ is a Krull domain
if $D$ is a locally finite intersection of localizations at height one
primes such that $D_{P}$ is a discrete valuation domain at each height one
prime. Thus a Krull domain is completely integrally closed. Glaz and
Vasconcelos \cite{GV 1977} called an integral domain $D$ an H-domain if for
an ideal $A$ with $A^{-1}=D,$ (or equivalently $A_{v}=D)$ then $A$ contains
a finitely generated subideal $F$ such that $A^{-1}=F^{-1}$. They showed
that a completely integrally closed H-domain is a Krull domain. In \cite[%
Proposition 2.4]{HZ 1988} it was shown that $D$ is an H-domain if and only
if every maximal $t$-ideal of $D$ is divisorial. We have in the following a
basic result and some of its derivatives.

\begin{proposition}
\label{Proposition UX1}(a) Let $D$ be a completely integrally closed domain.
Then (1) $D$ is a Krull domain if and only if $I_{t}(D)\vartriangleleft P$
for $P=\,$"--- is a proper divisorial ideal, (2) $D$ is a locally factorial
Krull domain if and only if $I_{t}(D)\vartriangleleft P$ for $P=\,$"--- is a
proper invertible integral ideal of $D,$ (3) $D$ is a Krull domain if and
only if $I_{t}(D)\vartriangleleft P$ for $P=\,$"--- is a proper $t$%
-invertible $t$-ideal of $D,$ (b) (4) Let $D$ be such that $D_{M}$ is a
Krull domain for each maximal ideal $M$ of $D.$ Then $D$ is a Krull domain
if and only if $I_{t}(D)\vartriangleleft P$ for $P=\,$"--- is a proper
divisorial ideal of $D$ \cite{EIT 2019} (5) Let $D$ be an intersection of
rank one valuation domains. Then $D$ is a Krull domain if and only if $%
I_{t}(D)\vartriangleleft P$ for $P=\,$"--- is a proper divisorial ideal of $%
D,$ (6) Let $D$ be an almost Dedekind domain. Then $D$ is a Dedekind domain
if and only if $I(D)\vartriangleleft P$ for $P=\,$"--- is a proper
divisorial ideals of $D.$
\end{proposition}

\begin{proof}
The idea of proof, in each case, is that every maximal $t$-ideal (maximal
ideal) being contained in a proper divisorial ideal must be equal to it and
combining this with the fact that $D$ is completely integrally closed we get
the Krull domain conclusion. For the locally factorial domain conclusion in
(2) we note that every maximal $t$-ideal of $D$ is invertible and so
divisorial. This gives the Krull conclusion and a Krull domain is locally
factorial if and only if every height one prime of $D$ is invertible \cite[%
Theorem 1]{And 1978}. For the Dedekind domain conclusion in (6), we note
that every maximal ideal is of height one and divisorial, being invertible,
so every maximal ideal is a $t$-ideal and so the domain is Krull and one
dimensional. The converse in each case is obvious, in that if $D$ is a Krull
domain then $D$ is completely integrally closed and every maximal $t$-ideal
of $D$ is, a $t$-invertible $t$-ideal and hence, divisorial. (If $D$ is
locally factorial, as in (2), every maximal $t$-ideal of $D$ is invertible
and hence divisorial.) And if $D$ is Dedekind, then $D$ is completely
integrally closed and every maximal ideal is invertible and hence divisorial.
\end{proof}

It is well known that $D$ is a Krull domain if and only if every $t$-ideal
of $D$ is a $t$-product of prime $t$-ideals of $D$ \cite{Nishi 1973}. As we
have seen, the prime $t$-ideals in a Krull domain happen to be all $t$%
-invertible $t$-ideals, and hence maximal $t$-ideals and divisorial \cite[%
Proposition 1.3]{HZ 1989}. Also, according to \cite[Theorem 1.10]{Zaf 1986}, 
$D$ is a locally factorial Krull domain if, and only if, every $t$-ideal of $%
D$ is invertible. Finally, $D$ being completely integrally closed may not
control the dimension of $D$ when every maximal ideal is a $t$-ideal.
Because the ring of entire functions is an infinite dimensional Bezout
domain and completely integrally closed \cite[page 146]{Gil 1972}. (Also, in
a Bezout domain every maximal ideal is a $t$-ideal.)

Another condition that helps control the dimension is requiring some kind of
an ascending chain condition. Call $D$ a $t$-ACC domain if $D$ satisfies ACC
on its $t$-invertible $t$-ideals.

\begin{lemma}
\label{Lemma UX2}Let $D$ be a $t$-ACC domain and let $I$ be a proper $t$%
-invertible $t$-ideal of $D.$ Then $\cap (I^{n})_{t}=(0).$ Consequently, in
a domain satisfying $t$-ACC, if $A$ is a proper divisorial ideal of $D$ and $%
I$ a $t$-invertible $t$-ideal then $(AI)_{v}=A$ implies $I=D$.
\end{lemma}

\begin{proof}
Because a $t$-invertible $t$-ideal is a $v$-ideal of finite type with $%
I^{-1} $ of finite type there is no harm in using $v$ for $t.$ Now let $\cap
(I^{n})_{v}\neq 0$ and let $x$ be a nonzero element in $\cap (I^{n})_{v}.$
Then there is a chain of $t$-invertible $t$-ideals $xI^{-1}\subseteq
(xI^{-2})_{v}\subseteq ...\subseteq x(I^{-r})_{v}$ $...$ which must stop
after a finite number of steps, because of the $t$-ACC restriction. Say $%
x(I^{-n})_{v}=x(I^{-n-1})_{v}.$ Cancelling $x$ from both sides we get $%
(I^{-n})_{v}=(I^{-n-1})_{v}.$ Multiplying both sides by $I^{n+1}$ and
applying the $v$-operation we get $I=D,$ a contradiction that arises from
assuming that there is a nonzero element in $\cap (I^{n})_{v}.$ For the
consequently part note that $(AI)_{v}=A$ implies that $A\subseteq
(I^{n})_{v} $ for all positive integers $n$.
\end{proof}

\begin{proposition}
\label{Proposition UX3}Let $D$ be a $t$-ACC domain. Then (1) $D$ is a PID if
and only if $I(D)\vartriangleleft P$ for $P=\,$"--- is a proper nonzero
principal ideal" and (2) $D$ is a Dedekind domain if and only if $%
I(D)\vartriangleleft P$ for $P=\,$"--- is a proper invertible ideal" and (3) 
$D$ is a Krull domain if and only if $I_{t}(D)\vartriangleleft P$ for $P=\,$%
"--- is a proper $t$-invertible $t$-ideal"$.$
\end{proposition}

\begin{proof}
We shall prove (3) and explain why it should work for the other two cases.
For (3) note that $I_{t}(D)\vartriangleleft P$ for $\Leftrightarrow \forall
A\in I_{t}(D)$ $(A\neq D$ $\Rightarrow \exists \mathcal{\gamma }\in \Gamma
(A\subseteq \mathcal{\gamma }))$ where $\Gamma $ is the set determined by $%
P=\,$"--- is a proper $t$-invertible $t$- (resp., nonzero principal,
invertible) ideal"$.$ Then, by the condition, every maximal $t$-ideal
(maximal ideal) of $D$ is $t$-invertible (resp., principal, invertible). By
Lemma \ref{Lemma UX2} we have for each maximal $t$-ideal $M$ (maximal ideal $%
M)$ $\cap (M^{n})_{v}=(0)$ (resp., $\cap M^{n}=(0),$ since powers of
principal (invertible) ideals are $v$-ideals). Thus each maximal $t$-ideal
(maximal ideal) is of height one. Thus $D$ is of $t$-dimension one (resp.,
of dimension one). Now, in each case, $MD_{M}$ is of height one and
principal, forcing $D_{M}$ to be a discrete rank one valuation domain for
each maximal $t$-ideal (maximal ideal) $M$. This makes $D$ completely
integrally closed, for $D=\cap D_{M}$ where $M$ ranges over maximal $t$%
-ideals (maximal ideals). Now apply Proposition \ref{Proposition UX1}, using
the fact that each maximal $t$-ideal (maximal ideal) is divisorial, being a $%
t$-invertible $t$-ideal (principal (invertible) ideal). The converse is
obvious in each case.
\end{proof}

\begin{proposition}
\label{Proposition UX4}Let $D$ be a $t$-ACC domain. Then (1) $D$ is a UFD if
and only if $I_{t}(D)\vartriangleleft P$ for $P=\,$"--- is a proper nonzero
principal ideal" and (2) $D$ is a locally factorial Krull domain if and only
if $I_{t}(D)\vartriangleleft P$ for $P=\,$"--- is a proper invertible ideal".
\end{proposition}

\begin{proof}
We shall prove (1) and explain why it should work for the other case. For
(1) note that $I_{t}(D)\vartriangleleft P$ for $P=\,$"--- proper nonzero
principal (invertible) ideal" $\Leftrightarrow \forall A\in I_{t}(D)$ $%
(A\neq D$ $\Rightarrow \exists \mathcal{\gamma }\in \Gamma (A\subseteq 
\mathcal{\gamma }))$ where $\Gamma $ is the set determined by $P$ returning $%
T.$ Then, by the condition, every maximal $t$-ideal of $D$ is principal
(invertible). By Lemma \ref{Lemma UX2} we have for each maximal $t$-ideal $M$%
, $\cap M^{n}=(0)$ $,$ since powers of principal (invertible) ideals are $v$%
-ideals. Thus each maximal $t$-ideal is of height one. Thus $D$ is of $t$%
-dimension one. Now, in each case, $MD_{M}$ is of height one and principal,
forcing $D_{M}$ to be a rank one valuation domain for each maximal $t$%
-ideal. This makes $D=\cap D_{M},$ where $M$ ranges over maximal $t$-ideals,
a completely integrally closed domain. Now apply Proposition \ref%
{Proposition UX1}, using the fact that each maximal $t$-ideal is divisorial,
being principal or invertible. This gets us the Krull conclusion. Now recall
that in a Krull domain $D,$ $A_{t}=(P_{1}^{n_{1}}...P_{r}^{n_{r}})_{t}.$
Then, in case of (2), $D$ is locally factorial by \cite[Theorem 1.10]{Zaf
1986} and, in case of (1), $D$ is factorial because every principal ideal is
a product of prime powers. The converse, in each case, is obvious in that a
UFD (locally factorial Krull domain) is Krull every maximal $t$-ideal of
whose is principal (resp., invertible).
\end{proof}

As already mentioned, an integral domain $D$ that satisfies ACC on integral
divisorial ideals is called a Mori domain. Obviously a Noetherian domain is
a Mori domain. It is easy to check that for every nonzero integral ideal $A$
of a Mori domain $D$ there are elements $a_{1},..,a_{r}\in A$ such that $%
A_{v}=(a_{1},..,a_{r})_{v}.$ So the inverse of a nonzero ideal of a Mori
domain is a $v$-ideal of finite type. Hence a $v$-invertible ideal in a Mori
domain is $t$-invertible. It is well known that a domain $D$ is a Krull
domain if, and only if, every nonzero ideal of $D$ is $t$-invertible (see
e.g. \cite[Theorem 2.5]{MZ 1991}) and thus a Krull domain is Mori too.
Noting that a Mori domain is a $t$-ACC domain and that Noetherian is Mori
too, we have the following direct corollaries.

\begin{corollary}
\label{Corollary UX5}Let $D$ be a Mori domain. Then (1) $D$ is a PID if and
only if $I(D)\vartriangleleft P$ for $P=\,$"--- is a proper nonzero
principal ideal", (2) $D$ is a Dedekind domain if and only if $%
I(D)\vartriangleleft P$ for $P=\,$"--- is a proper invertible ideal", (3) $D$
is a Krull domain if and only if $I_{t}(D)\vartriangleleft P$ for $P=\,$"---
is a proper $t$-invertible $t$-ideal"$,$ (4) $D$ is a UFD if and only if $%
I_{t}(D)\vartriangleleft P$ for $P=\,$"--- is a proper nonzero principal
ideal" and (5) $D$ is a locally factorial Krull domain if and only if $%
I_{t}(D)\vartriangleleft P$ for $P=\,$"--- is proper invertible ideal".
\end{corollary}

\begin{corollary}
\label{Corollary VX}Let $D$ be a Noetherian domain. Then (1) $D$ is a PID if
and only if $I(D)\vartriangleleft P$ for $P=\,$"--- is a proper nonzero
principal ideal" and (2) $D$ is a Dedekind domain if and only if $%
I(D)\vartriangleleft P$ for $P=\,$"--- is a proper invertible ideal".
\end{corollary}

\begin{corollary}
\label{Corollary WX}Let $D$ be a Mori domain. Then (1) $D$ is a UFD if and
only if $I_{t}(D)\vartriangleleft P$ for $P=\,$"--- is a proper nonzero
principal ideal", (2) $D$ is a locally factorial Krull domain if and only if 
$I_{t}(D)\vartriangleleft P$ for $P=\,$"--- is a proper invertible integral
ideal"$,$ (3) $D$ is a Krull domain if and only if $I_{t}(D)\vartriangleleft
P$ for $P=\,$"--- is a proper $t$-invertible $t$-ideal"$.$
\end{corollary}

Finally, consider the following scheme of results.

\begin{proposition}
\label{Proposition XX}Suppose that $D$ satisfies ACCP (ACC on principal
ideals). Then (1) $D$ is a PID if and only if $I(D)\vartriangleleft P$ for $%
P=\,$"--- is a proper nonzero principal ideal" and (2) $D$ is a UFD if and
only if $I_{t}(D)\vartriangleleft P$ for $P=\,$"--- is a proper nonzero
principal ideal"$.$
\end{proposition}

\begin{proof}
$I_{t}(D)\vartriangleleft P$ for $P=\,$"--- is a proper nonzero principal
ideal"$\Leftrightarrow \forall A\in I_{\ast }(D)$ $(A\neq D$ $\Rightarrow
\exists \mathcal{\gamma }\in \Gamma (A\subseteq \mathcal{\gamma }))$ where $%
\Gamma $ is the set of proper nonzero principal ideals of $D$ fixed by $P$
and $\ast =d$ or $t$. Then, by the condition, for any maximal (maximal $t$%
-ideal) $M$ of $D$ we have $M\subseteq \mathcal{\gamma }$ for some $\mathcal{%
\gamma }\in \Gamma $ and so $M=$ $\gamma D.$ Claim that, because of the
ACCP, $M$ is of height one. (For if not, then there is $Q\subseteq \cap 
\mathcal{\gamma }^{n}D$. So for every nonzero $x\in Q,$ $x$ is divisible by
every power of $\mathcal{\gamma },$ giving rise to an infinite ascending
chain $xD\subsetneq \frac{x}{\mathcal{\gamma }}D\subsetneq \frac{x}{\mathcal{%
\gamma }^{2}}\subsetneq ...\subsetneq \frac{x}{\mathcal{\gamma }^{n}}%
D\subsetneq ...$ which is impossible in the presence of ACCP on $D.)$ Now $%
MD_{M}$ is principal and of height one, making $D_{M}$ a rank one discrete
valuation domain and making $D=\cap D_{M}$ completely integrally closed with
every maximal ($t$-) ideal principal. This makes $D$ a Krull domain with
every height one prime a principal ideal and so a UFD. Finally, a UFD with
every height one prime maximal is a PID. The converse, in each case is
straightforward.
\end{proof}

\section{\protect\bigskip A Universal Restriction with Conditions \label%
{Section S3}}

Call a directed p.o. group $G$ an almost l.o. group if for each finite
subset $X=\{x_{1},...,x_{r}\}\subseteq G^{+}$ there is a positive integer $%
n=n(X)$ such that $inf(x_{1}^{n},...x_{r}^{n})\in G^{+}.$ Almost l.o. groups
were introduced in \cite{DLMZ 2001} and further studied in \cite{Yang 2008}.
One can talk about a commutative p.o. monoid $M$ with least element $1$ and
a pre-assigned set $\Gamma $ such that for all $x_{1},...,x_{r}\in M$ with $%
\mathcal{L}(x_{1},...,x_{r})\neq 1,$ there being a $\mathcal{\gamma }\in
\Gamma $ such that $x_{1}^{n},...,x_{r}^{n}\geq \mathcal{\gamma }.$ As ring
theory provides a plethora of examples of this concept, we turn to ring
theory.

Let $D$ be a domain with a finite type star operation $\ast $ defined on it,
let $I_{\ast }(D)$ be the set of proper $\ast $-ideals of $D$ and let $%
\Gamma _{I_{\ast }(D)}(P)$ be a non-empty subset of $I_{\ast }(D)$ defined
by a predicate $P$ such that for each $A\in $ $I_{\ast }(D)$ there is an
integer $n=n(A)\geq 1$ with $A^{n}\subseteq \mathcal{\gamma }$ for some $%
\mathcal{\gamma }\in \Gamma _{I_{\ast }(D)}(P).$ Let us say $I_{\ast }(D)$
meets $P$ with a twist when this happens and denote it by $I_{\ast
}(D)\vartriangleleft ^{t}P.$ We start with a motivating example of this
notion.

\begin{example}
\label{Example A}Let $D$ be a Dedekind domain with torsion class group, let $%
K$ be the quotient field of $D$ and let $X$ be an indeterminate over $K.$
Then the ring $R=D+XK[X]$ is such that $I(R)\vartriangleleft ^{t}P$ where $%
P= $ "--- is a principal ideal"$.$
\end{example}

Illustration: Recall, as we have already done, from \cite[Theorem 4.21]{CMZ
1978} that maximal ideals of $R$ are of the form $M+XK[X],$ where $M$ is a
maximal ideal of $D,$ or of the form $(1+Xf(X))R$ where $1+Xf(X)$ is
irreducible in $K[X].$ Now since for each maximal ideal $M$ of $D$ we have $%
M^{n}\subseteq dD$ for some positive integer $n$ and some nonzero $d\in D$
we have $(M+XK[X])^{n}=$ $M^{n}+XK[X]\subseteq dD+XK[X].$ Next since for
each maximal ideal $\mathcal{M}$ of $R$, either $\mathcal{M}$ is principal,
and hence is contained in a principal ideal in $\Gamma _{I_{\ast }(R)}(P)$
or $\mathcal{M}$ is such that $\mathcal{M}^{n}$ is contained in a principal
ideal for some positive integer $n,$ the same must hold for every ideal $I$
of $R$.

The above example leads to the following statement.

\begin{proposition}
\label{Proposition A1}(1)$I(D)\vartriangleleft ^{t}P$ where $P=$ "---- is a
proper nonzero finitely generated ideal" if and only if for every maximal
ideal $M$ of $D$ we have $M^{n}\subseteq \mathcal{\gamma }\in \Gamma
_{I(D)}(P)$ and (2) let $L$ be an extension field of $K=qf(D)$ and let $X$
be an indeterminate over $L.$ Then $I(D)\vartriangleleft ^{t}P$ where $P=$
"---- is a proper nonzero finitely generated ideal" if and only if $%
I(R)\vartriangleleft ^{t}P$ where $R=D+XL[X].$
\end{proposition}

\begin{proof}
(1) Suppose that for every ideal $A$ of $D$ we have some integer $n=n(A)\geq
1$ and a $\mathcal{\gamma }_{A}\in \Gamma _{I(D)}$ such that $A^{n}\subseteq 
\mathcal{\gamma }_{A}$ then the same holds if $A$ is a maximal ideal of $D.$
Conversely suppose that for every maximal ideal $M$ of $D$ we have some
integer $n=n(M)\geq 1$ and some $\mathcal{\gamma }_{M}\in \Gamma _{I(D)}$
such that $M^{n(M)}\subseteq \mathcal{\gamma }_{M}$ and let $A$ be a proper
nonzero ideal of $D.$ Then $A\subseteq M$ for some maximal ideal $M$ of $D$
and $A^{n(M)}\subseteq M^{n(M)}\subseteq \mathcal{\gamma }_{M}.$ For (2) let 
$I(D)\vartriangleleft ^{t}P$, where $P$ is as given, then, by (1), for every
maximal ideal $\wp $ of $D$ there is $n=n(\wp )$ such that $\wp
^{n}\subseteq \mathcal{\gamma }_{\wp }$ for some $\mathcal{\gamma }_{\wp
}\in \Gamma _{I(D)}.$ Since every maximal ideal of $D+XL[X]$ is either
principal, and hence finitely generated, or of the form $\wp +XL[X]$ where $%
\wp $ is a maximal ideal of $D$ \cite[Lemmas 1.2, 1.5]{CMZ 1986}, for every
ideal $A$ of $R$ there is $n=1$ or $n(\wp )\geq 1$ such that $A^{n}\subseteq 
\mathcal{\gamma }_{A}\in \Gamma _{I(R)},$ so $I(R)\vartriangleleft ^{t}P.$
Conversely suppose that $I(R)\vartriangleleft ^{t}P.$ Then, in particular,
for maximal ideals $\mathcal{M}$ of the form $\wp +XL[X]$ there are positive
integers $n(\mathcal{M})$ such that $(\wp +XL[X])^{n(\mathcal{M})}$ $=\wp
^{n(\mathcal{M})}+XL[X]$ $\subseteq \mathcal{\gamma }_{\mathcal{M}}\in
\Gamma _{I(R)}.$ But then $\mathcal{\gamma }_{\mathcal{M}}\cap D\neq (0)$
forcing $\mathcal{\gamma }_{\mathcal{M}}=\mathcal{\gamma }+XL[X]=\mathcal{%
\gamma }(D+XL[X])$ \cite[Lemma1.1]{CMZ 1986}, where $\mathcal{\gamma }$ is
finitely generated because $\mathcal{\gamma }_{\mathcal{M}}$ is. This gives $%
\wp ^{n(\mathcal{M})}+XL[X]\subseteq \mathcal{\gamma }+XL[X]$ and modding
out $XL[X]$ we get $\wp ^{n(\mathcal{M})}\subseteq \mathcal{\gamma }\in
\Gamma _{I(D)}=\{\mathcal{\gamma }\neq (0)|\mathcal{\gamma }+XL[X]\in \Gamma
_{I(R)}\}.$
\end{proof}

\begin{proposition}
\label{Proposition A2}(1) $I(D)\vartriangleleft ^{t}P$ where $P=$ "--- is a
proper $t$-ideal of finite type" if and only if for every maximal ideal $M$
of $D,$ there is $n=n(M)\geq 1$ such that $M^{n}\subseteq \mathcal{\gamma }%
\in \Gamma _{I(D)}(P),$ (2) $I_{t}(D)\vartriangleleft ^{t}P$ where $P=$ "---
is a proper $t$-ideal of finite type" if and only if for every maximal $t$%
-ideal $M$ of $D$ we have $M^{n}\subseteq \mathcal{\gamma }\in \Gamma
_{I(D)}(P)$, (3) let $L$ be an extension field of $K$ and let $X$ be an
indeterminate over $L.$ Then $I(D)\vartriangleleft ^{t}P$ where $P=$ "--- is
a proper $t$-ideal of finite type" if and only if $I(R)\vartriangleleft
^{t}P $ where $R=D+XL[X]$ and (4) let $L$ be an extension field of $K$ and
let $X$ be an indeterminate over $L.$ Then $I_{t}(D)\vartriangleleft ^{t}P$
where $P= $ "--- is a proper $t$-ideal of finite type" if and only if $%
I_{t}(R)\vartriangleleft ^{t}P$ where $R=D+XL[X],$ (5) Repeat (3) and (4)
for $R=D+XL[[X]].$
\end{proposition}

\begin{proof}
(1) The proof works as the proof of (1) of Proposition \ref{Proposition A1}.
(2) The proof works in the same manner as that of (1) of Proposition \ref%
{Proposition A1}, except that here the maximal $t$-ideals are in the focus.
(3) Let $I(D)\vartriangleleft ^{t}P$ where $P=$ "--- is a proper $t$-ideal
of finite type". To show that $I(R)\vartriangleleft ^{t}P$ all we need show
is that for every maximal ideal $\mathcal{M}$ of $R$, there is a positive
integer $n=n(\mathcal{M)}$ such that $\mathcal{M}^{n}\subseteq \mathcal{%
\gamma }_{\mathcal{M}}\in \Gamma _{I(R)}(P).$ Now, as we have shown in the
proof of (2) of Proposition \ref{Proposition A1}, a maximal ideal $\mathcal{M%
}$ of $R$ is either principal and hence contained in some member of $\Gamma
_{I(R)}$ or of the form $\mathcal{M}=M+XL[X],$ where $M$ is a maximal ideal
of $D.$ But then, for $n=n(M)$ we have $M^{n}\subseteq \gamma ,$ where $%
\mathcal{\gamma }$ is a $t$-ideal of finite type in $\Gamma _{I(D)}$,
forcing $\mathcal{M}^{n}=M^{n}+XL[X]\subseteq \gamma R.$ Because $\mathcal{%
\gamma }$ is a $t$-ideal of finite type of $D,$ so is $\gamma R=\mathcal{%
\gamma }+XL[X],$ see e.g. proof of Lemma 3.5 of \cite{Zaf 2019}. Conversely,
suppose that $I(R)\vartriangleleft ^{t}P$ where $P=$ "--- is a proper $t$%
-ideal of finite type"$.$ Here, in particular, for a maximal ideal $\mathcal{%
M}$ of the form $\mathcal{M}=M+XL[X]$ we have a positive integer $n=n(%
\mathcal{M})$ such that $\mathcal{M}^{n(\mathcal{M)}}=M^{n(\mathcal{M)}%
}+XL[X]$ $\subseteq \mathcal{\gamma }_{\mathcal{M}}$ where $\mathcal{\gamma }%
_{\mathcal{M}}$ is a $t$-ideal of finite type of $R.$ Obviously as $M^{n(%
\mathcal{M)}}=(M^{n(\mathcal{M)}}+XL[X])\cap D\subseteq \mathcal{\gamma }_{%
\mathcal{M}}\cap D,$ and as $M^{n(\mathcal{M)}}\cap D\neq (0),$ we conclude
that $\mathcal{\gamma }_{\mathcal{M}}\cap D\neq (0).$ Thus $\mathcal{\gamma }%
_{\mathcal{M}}=\mathcal{\gamma }_{\mathcal{M}}\cap D+XL[X]$ by \cite[Lemma
1.1]{CMZ 1986}. And as observed in the proof of Lemma 3.5 of \cite{Zaf 2019} 
$\mathcal{\gamma }_{\mathcal{M}}=\mathcal{\gamma }_{\mathcal{M}}\cap D+XL[X]$
is a $t$-ideal of $R$ if and only if $\mathcal{\gamma }_{\mathcal{M}}\cap D$
is a $t$-ideal of $D.$ That $\mathcal{\gamma }_{\mathcal{M}}$ is of finite
type if and only if $\mathcal{\gamma }_{\mathcal{M}}\cap D$ is follows from
the fact that $\mathcal{\gamma }_{\mathcal{M}}=(a_{1},...,a_{n})_{v}+XL[X].$
Finally, for (4), let $I_{t}(D)\vartriangleleft ^{t}P$ where $P=$ "--- is a
proper $t$-ideal of finite type" and as maximal $t$-ideals of $R$ that
intersect $D$ trivially are prime ideals of $R$ that intersect $D$
trivially, are not comparable with $XL[X],$ and hence are principal we need
to concentrate on maximal $t$-ideals $\mathcal{M}$ of $R$ that intersect $D$
non-trivially. But those are precisely $\mathcal{M=(M\cap }$ $D)+XL[X]$ and
as $\mathcal{M=M}_{t}\mathcal{=(M\cap }$ $D)_{t}~+XL[X]$ we have $\mathcal{%
(M\cap }$ $D)_{t}~=\mathcal{(M\cap }$ $D).~$Thus $\mathcal{M}=M+XL[X]$ where 
$M$ is a maximal $t$-ideal of $D.$ But, by the hypothesis, there is a
positive integer $n=n(M)$ such that $M^{n(M)}\subseteq \mathcal{\gamma }_{M}$
for some $\mathcal{\gamma }_{M}\in \Gamma _{I_{t}(D)}(P).$ This forces $%
\mathcal{M}^{n(M)}=M^{n(M)}+XL[X]\subseteq \mathcal{\gamma }_{M}+XL[X]$
which is a $t$-ideal of finite type and hence in $\Gamma _{I(R)}.$ For the
converse we take the same line as in the proof of (3) and note that for each
maximal $t$-ideal $M$ of $D,$ $\mathcal{M=}M+XL[X]$ is a maximal $t$-ideal
of $R$ and as $\mathcal{M}^{n}=(M+XL[X])^{n}\subseteq \mathcal{\gamma }_{%
\mathcal{M}}$ for some $\mathcal{\gamma }_{\mathcal{M}}\in \Gamma
_{I_{t}(R)}(P),$ $M^{n}\subseteq \mathcal{\gamma }_{\mathcal{M}}\cap D\neq
(0).$ Now, as in (3), $\mathcal{\gamma }_{\mathcal{M}}\cap D$ can be shown
to be a $t$-ideal of finite type and hence in $\Gamma _{I_{t}(D)}(P).$ For
(5) note that each maximal ($t$-) ideal of $R$ is of the form $m+XL[[X]]$
where $m$ is a maximal ($t$-) ideal of $D.$
\end{proof}

Apart from the examples constructed in the above proposition there are
examples of domains $I_{\ast }(D)\vartriangleleft ^{t}P$ for $P=$ "--- is a $%
\ast $-ideal of finite type". Some of these examples are simple and
straightforward and some are not so simple. Presented in the following is a
sampling of them. If $D$ is Noetherian and $P=$ "--- is a finitely generated
ideal, then $I(D)\vartriangleleft ^{t}P.$ Recall, again, that $D$ is a Mori
domain if it satisfies ACC on its integral divisorial ideals. Obviously
Noetherian domains are Mori and less obviously Krull domains are Mori.
Recall also that $D$ is Mori if and only if for every nonzero integral ideal 
$A$ of $D$ there is a finitely generated ideal $F\subseteq A$ such that $%
A_{v}=F_{v},$ if and only if every $t$-ideal of $D$ is a $t$-ideal of finite
type \cite{Zaf 1989}. Thus if $D$ is a Mori domain then $I_{t}(D)%
\vartriangleleft ^{t}P$ where $P=$ "--- is a $t$-ideal of finite type". Note
that since for a finitely generated nonzero ideal $A$ of any domain $%
A_{t}=A_{v}$, every $t$-ideal of a Mori domain is divisorial. In what
follows we shall also need the fact that if $I$ is a $\ast $-ideal for some
star operation $\ast $, then $\sqrt{I}$ is a $\ast _{s}$-ideal (see Theorem
1 of \cite{Zaf 2005}). Thus if $I$ is divisorial, or a $t$-ideal then $\sqrt{%
I}$ is a $t$-ideal.

\begin{proposition}
\label{Proposition A3}Let $D$ be a Mori domain. Then $I(D)\vartriangleleft
^{t}P$ with $P=$ "--- is a $t$-ideal" if and only if every maximal ideal of $%
D$ is divisorial.
\end{proposition}

\begin{proof}
If every maximal ideal $M$ of $D$ is a $t$-ideal then, $D$ being Mori, $M$
is a $t$-ideal of finite type, hence divisorial and hence in $\Gamma
_{I(D)}(P)$. Whence $I(D)\vartriangleleft ^{t}P$. Conversely suppose that $D$
is Mori and $I(D)\vartriangleleft ^{t}P$ where $P$ is as given and let $M$
be a maximal ideal of $D.$ Then by the condition $M^{n}\subseteq A$ where $A$
is a $t$-ideal. This gives $M=\sqrt{M^{n}}\subseteq \sqrt{A}.$ Since $M$ is
maximal, we have $M=\sqrt{A}$ which is a $t$-ideal. Since $M$ is arbitrary
we have the result.
\end{proof}

The event of $I(D)\vartriangleleft ^{t}P$ for $P=$ "---- is a $t$-ideal of
finite type" does not put any constraint on the height of maximal ideals of
a Mori domain. Indeed there do exist examples of Noetherian domains with
maximal $t$-ideals of height greater than one, see e.g. \cite[Example 3.5]%
{FZ 2019}.

\begin{corollary}
\label{Corollary A4}Let $D$ be a Noetherian integral domain. Then $%
I(D)\vartriangleleft ^{t}P$ with $P=$ "---- is a $t$-ideal of finite type"
if and only if every maximal ideal of $D$ is divisorial.
\end{corollary}

Indeed as in a polynomial ring over $D\neq K$, every maximal ideal being a
radical of a $t$-ideal of any kind is not possible because that would make
every maximal ideal of the polynomial ring a $t$-ideal as we have seen in
section 2. On the other hand, we have the following statement.

\begin{proposition}
\label{Proposition A4}Let $R=D[X]$. If $P=$ "---- is a $t$-invertible $t$%
-ideal (resp., divisorial ideal)" Then $I_{t}(D)\vartriangleleft
^{t}P\Rightarrow I_{t}(R)\vartriangleleft ^{t}P$ and if $D$ is integrally
closed, $I_{t}(R)\vartriangleleft ^{t}P\Rightarrow I_{t}(D)\vartriangleleft
^{t}P$.
\end{proposition}

\begin{proof}
(a). Let $M$ be a maximal $t$-ideal of $D[X]$ and suppose that $M\cap D\neq
(0).$ Then $M=\wp \lbrack X]$ where $\wp =M\cap D$ is a maximal $t$-ideal of 
$D$ \cite{HZ 1989}. Since $I_{t}(D)\vartriangleleft ^{t}P$ we conclude that
for some $n=n(\wp )\geq 1,$ $\wp ^{n}$ is contained in a $t$-invertible $t$%
-ideal (resp. $t$-ideal, divisorial ideal) $A$. But then, $M^{n}=\wp
^{n}[X]\subseteq A[X].$ Next let $M$ be a maximal $t$-ideal of $D[X]$ such
that $M\cap D=(0).$ Then $M$ is a $t$-invertible $t$-ideal and hence
divisorial by Theorem 1.4 of \cite{HZ 1989} and $M^{n}\subseteq M$ for all $%
n.$ Next suppose that $I_{t}(R)\vartriangleleft ^{t}P$ for the specified $P.$
Then, in particular, for every maximal $t$-ideal $\wp $ of $D$ we have the
maximal $t$-ideal $M=\wp \lbrack X]$ and, by the condition, there is $%
n=n(M)\geq 1$ such that $M^{n}$ is contained in a $t$-ideal (resp., $t$%
-invertible $t$-ideal, divisorial ideal) $A$ of $D[X].$ Since $M^{n}\cap
D\neq (0),$ $A\cap D\neq (0)$ and since $D$ is integrally closed $A=(A\cap
D)[X]$ and $A\cap D$ is a $t$-ideal (resp., $t$-invertible $t$-ideal,
divisorial ideal), if $A$ is \cite[Corollary 3.1]{AKZ 1995}.
\end{proof}

Indeed as the behavior of $D+XL[X]$ is the same under $S(D)\vartriangleleft
^{t}P$ as it was under $S(D)\vartriangleleft P$, one can construct examples
to show that if $R$ is a ring of fractions of $D,$ $S(D)\vartriangleleft
^{t}P$ may not imply $S(R)\vartriangleleft ^{t}P$ in general. This leaves us
to check what happens if we restrict a domain to be completely integrally
closed and satisfy $S(D)\vartriangleleft ^{t}P$ for a suitable $P.$ To
appreciate the following proposition we need to have an idea of the divisor
class group of a Krull domain being torsion. For this too the reference to
go to is \cite{Fos 1973}. For our purposes the divisor class group being
torsion means that for each proper divisorial ideal $I$ there is some
positive integer $n$ such that $(I^{n})_{v}$ is principal. The other concept
to know is the local class group $G(D)=Cl(D)/Pic(D)$ of a Krull domain $D,$
introduced and studied by Bouvier in \cite{Bou 1983}. Now $G(D)$ being
torsion is equivalent to $(I^{n})_{v}$ being invertible, for some integer $%
n\geq 1,$ for each proper divisorial ideal $I.$

\begin{proposition}
\label{Proposition A5}(a) Let $D$ be a completely integrally closed domain.
Then (1) $D$ is a Krull domain if and only if $I_{t}(D)\vartriangleleft
^{t}P $ for $P=\,$"--- is a proper divisorial ideal", (2) $D$ is a Krull
domain if and only if $I_{t}(D)\vartriangleleft ^{t}P$ for $P=\,$"--- is a
proper $t$-invertible $t$-ideal", (3) $D$ is a Krull domain, with torsion
divisor class group, if and only if $I_{t}(D)\vartriangleleft ^{t}P$ for $%
P=\,$"--- is a proper principal ideal"$.$ (b) Let $D$ be an intersection of
rank one valuation domains. Then (4) $D$ is a Krull domain, if and only if $%
I_{t}(D)\vartriangleleft ^{t}P$ for $P=\,$"--- a proper $v$-ideal of finite
type" and (5) $D$ is a Krull domain, with torsion local class group, if and
only if $I_{t}(D)\vartriangleleft ^{t}P$ for $P=\,$"--- a proper invertible
ideal". (c) Let $D$ be completely integrally closed. Then (6) $D$ is a
Dedekind domain if and only if $I(D)\vartriangleleft ^{t}P$ for $P=\,$"---
is a proper divisorial ideal" (resp. invertible ideal) and (7) $D$ is a
Dedekind domain with torsion class group if and only if $I(D)%
\vartriangleleft ^{t}P$ for $P=\,$"--- is a proper principal ideal".
\end{proposition}

\begin{proof}
(1). Let $D$ be a completely integrally closed domain and let $%
I_{t}(D)\vartriangleleft ^{t}P$ for $P=\,$"--- is a proper divisorial
ideal". (I.e. suppose that for every $t$-ideal $I$ there is $n\geq 1$ such
that $I^{n}$ is contained in a divisorial ideal.) Now let $M$ be a maximal $%
t $-ideal of $D.$ We claim that $M$ is divisorial, for if not then $M_{v}=D.$
But, by the condition, $M^{n}$ is contained in a proper divisorial ideal $%
\mathcal{\gamma }.$ Thus $(M^{n})_{v}\subseteq \mathcal{\gamma }$ because $%
\mathcal{\gamma }$ is a divisorial ideal. On the other hand $%
(M^{n})_{v}=((M_{v})^{n})_{v}=D,$ contradicting the assumption that $%
\mathcal{\gamma }$ is a proper divisorial ideal. Whence $M_{v}\neq D,$
forcing $M=M_{v}.$ Now as $M$ is arbitrary, we conclude that $D$ is an H
domain \cite{HZ 1988}. Finally, according to \cite{GV 1977}, $D$ is Krull.
Conversely if $M$ is a maximal $t$-ideal of a Krull domain then $M$ is
divisorial and so is $(M^{n})_{v}$ which returns $T$ for $P$ for any $n$.
(2). Because a proper $t$-invertible $t$-ideal is divisorial too and because
every prime $t$-ideal of a Krull domain is $t$-invertible and so must be
every maximal $t$-ideal $M,$ with $(M^{n})_{v}$ a $t$-invertible $t$-ideal,
we conclude that the proof of (1) applies. (3). For sufficiency, note that a
proper principal ideal is divisorial. So $D$ is at least a Krull domain, by
part (1). Now let $M$ be a maximal $t$-ideal of $D.$ Then, by the condition, 
$M^{n}$ is contained in a proper nonzero principal ideal $\mathcal{\gamma }$
and clearly $M^{n}\subseteq \mathcal{\gamma }\subseteq M.$ Thus $M$ is the
radical of a principal ideal and Theorem 3.2 of \cite{And 1982} applies to
give the conclusion that the divisor class group of $D$ is torsion.
Conversely if $D$ is a Krull domain whose divisor class group is torsion,
then via Theorem 3.2 of \cite{And 1982} (or via \cite[Proposition 6.8]{Fos
1973}) one finds that for each maximal $t$-ideal $M$ we have $(M^{n})_{v}=%
\mathcal{\gamma }$ a principal ideal verifying that $M^{n}$ is contained in
a proper principal ideal for each maximal $t$-ideal $M$ of $D.$ Note in part
(b) that $D$ being completely integrally closed is provided by the given.
Then (4) can be proved just like (1) and that leaves (5). Now in (5) we
prove just like (3) that $D$ is a Krull domain and then use the condition to
show that $M$ is the radical of an invertible ideal. This would give, via
Theorem 3.3 of \cite{And 1982} the conclusion that $G(D)$ is torsion. For
necessity in this case we appeal to Theorem 3.3 of \cite{And 1982} to
conclude that $I_{t}(D)\vartriangleleft ^{t}P.$ For (6) and (7) note that
every maximal $t$-ideal is maximal, and divisorial, because every maximal
ideal is divisorial. So, in each cas\text{e, }$D$ is a one dimensional Krull
domain and hence a Dedekind domain. Now in case of (7) we can conclude, as
in the proof of (3), that every maximal ideal is the radical of a principal
ideal. The converse in each case is obvious, if not dealt with.
\end{proof}

For a star operation $\ast $ of finite type, defined on $D,$ call $D$ of
finite $\ast $-character if every nonzero non unit of $D$ belongs to at most
a finite number of maximal $\ast $-ideals of $D.$ We shall be mostly
concerned with $\ast =t$ or $d$ though some of the considerations here may
apply to the general approach. In any case we may define $\ast $-dimension
as the supremum of the lengths of chains of $\ast $-ideals that are prime.
Call $D$ a weakly Krull domain (WKD) if $D=\cap _{P\in X^{1}(D)}D_{P}$ and
the intersection is locally finite. It turns out that $D$ is of finite $t$%
-character and of $t$-dimension one \cite{AMZ 1992}. We shall also need to
use the $nth$ symbolic power $Q^{(n)}$of a prime $Q$ defined by $%
Q^{(n)}=Q^{n}D_{Q}\cap D$. We shall need also to recall that a nonzero
finitely generated ideal $I$ is said to be rigid ($t$-rigid) if $I$ is
contained in a unique maximal ($t$-) ideal. A maximal ($t$-) ideal is said
to be ($t$-) potent if it contains a ($t$-) rigid ideal. Finally a domain $D$
is said to be ($t$-) potent if each of its maximal ($t$-) ideals is $(t$-)
potent.

\begin{proposition}
\label{Proposition B}(1)Let $I(D)\vartriangleleft ^{t}P$ where $P=$ "--- is
a proper nonzero principal ideal" (resp. invertible ideal, $t$-invertible $t$%
-ideal) $.$ If $D$ has $t$-ACC, then $D$ is a $t$-potent domain whose
maximal ideals $M$ are divisorial such that $\cap (M^{n})_{v}=(0)$. (2) Let $%
I_{t}(D)\vartriangleleft ^{t}P$ where $P=$ "--- is a proper nonzero
principal ideal" (resp. invertible ideal, $t$-invertible $t$-ideal) $.$ If $%
D $ has $t$-ACC, then $D$ is a $t$-potent domain whose maximal $t$-ideals $M$
are divisorial such that $\cap (M^{n})_{v}=(0).$
\end{proposition}

\begin{proof}
For (1) let $I(D)\vartriangleleft ^{t}P$ where $P=$ "--- is a proper nonzero
principal ideal" (resp. invertible ideal, $t$-invertible $t$-ideal) and
suppose that $D$ has $t$-ACC. As we concluded in the proof of Proposition %
\ref{Proposition A5}, every maximal ideal $M$ is divisorial. Next, for every
maximal ideal $M$ we have $M^{n}\subseteq \mathcal{\gamma }\in \Gamma
_{I(D)}(P).$ This shows also that $M$ is $t$-potent. Next $%
(M^{n})_{v}\subseteq \mathcal{\gamma },$ because $\mathcal{\gamma }$ is
divisorial. So $\cap (M^{nr})_{v}\subseteq \cap (\mathcal{\gamma }^{n})_{v}.$
Since $\mathcal{\gamma }$ is a $t$-invetible $t$-ideal and since $D$ is $t$%
-ACC, Lemma \ref{Lemma UX2} applies to give $\cap (\mathcal{\gamma }%
^{n})_{v}=(0).$ Whence $\cap (M^{n})_{v}=(0).$ For (2) note that $%
I_{t}(D)\vartriangleleft ^{t}P$ implies that $M^{n}\subseteq \mathcal{\gamma 
}\in \Gamma _{I_{t}(D)}(P)$ for each maximal $t$-ideal $M$. Since $\mathcal{%
\gamma }$ is divisorial, $M$ must be. The rest of the proof follows the same
lines as taken in the proof of (1).
\end{proof}

The above result does not give much. But with some give and take it can.

\begin{proposition}
\label{Proposition B1} (a) Let $I(D)\vartriangleleft ^{t}P$ where $P=$ "---
is a proper nonzero principal ideal" and suppose that $D$ has $t$-ACC. Then
the following are equivalent: (1) $D$ is one dimensional, (2) for every
maximal ideal $M,$ $M^{n}$ being contained in a principal ideal $dD$ implies 
$Q^{(n)}\subseteq dD$ for every nonzero prime $Q$ contained in $M$, (3) $D$
is a one dimensional WKD and (4) every power of every nonzero prime ideal $Q$
of $D$ is a primary ideal. (b) Let $I_{t}(D)\vartriangleleft ^{t}P$ where $%
P= $ "--- is a proper nonzero principal ideal" and suppose that $D$ has $t$%
-ACC. Then the following are equivalent: (1) $D$ has $t$-dimension one, (2)
for every maximal $t$-ideal $M,$ $M^{n}$ being contained in a principal
ideal $dD$ implies $Q^{(n)}\subseteq dD$ for every nonzero prime $Q$
contained in $M$, (3) $D$ is a WKD.
\end{proposition}

\begin{proof}
(a) That (1) $\Rightarrow $ (2) is clear. For (2) $\Rightarrow $ (3), we
show that $D$ is one dimensional. Assume by way of contradiction that there
is a nonzero non-maximal prime $Q$ contained in a maximal ideal $M.$ Let $%
M^{n}\subseteq dD$ for a non unit $d\in D$ and let $0\neq x\in Q^{(n)}.$
Then $x\in dD.$ Since $d\notin Q,$ $(x/d)d\in Q^{(n)}$ forces $x/d\in
Q^{(n)}.$ Repeating the argument over and over again we get $\frac{x}{d}%
D\subseteq \frac{x}{d^{2}}D\subseteq \frac{x}{d^{3}}D\subseteq ...\subseteq 
\frac{x}{d^{n}}D\subseteq \frac{x}{d^{n+1}}D\subseteq ...$ which is
impossible in the presence of $t$-ACC. Thus $D$ is one dimensional and hence
of $t$-dimension one. Now a $t$-potent domain of $t$-dimension one is a WKD
by \cite[Theorem 5.3]{HZ 2019}. That (3) $\Rightarrow $ (4), is direct
because $D$ is one dimensional. For (4) $\Rightarrow $ (1), suppose that
there is a nonzero non-maximal prime ideal $Q$ and proceed as in the proof
of (2) $\Rightarrow $ (3) to get the desired contradiction. For the proof of
(b) note that (1) $\Rightarrow $ (2) is obvious and (2) $\Rightarrow $ (3)
goes exactly along the lines taken in the proof of (2) $\Rightarrow $ (3) of
(a), while (3) $\Rightarrow $ (1) is obvious too.
\end{proof}

Lest a reader considers Proposition \ref{Proposition B1} an empty result we
hasten to give examples to allay such feelings. For the following set of
examples we need to know that an extension of domains $A\subseteq B$ is
called a root extension if for each $b\in B$ there is a positive integer $%
n=n(b)$ such that $b^{n}\in A.$ Let's call $A\subseteq B$ a fixed root
extension if there is a fixed positive integer $n$ such that $b^{n}\in A,$
for all $b\in B.$ Also an integral domain $D$ is called an Almost Principal
Ideal (API-)domain if for each subset $\{a_{\alpha }\}$ of $D\backslash
\{0\} $ there is a positive integer $n$ such that $(\{a_{\alpha }^{n}\})$ is
principal. According to \cite[Theorem 4.11]{AZ 1991} if $A\subseteq B$ is a
fixed root extension and $B$ is a subring of the integral closure of $A$,
then $A$ is an API domain if and only if $B$ is.

\begin{example}
\label{Example B2}Of course (1) every Dedekind domain $D$ with torsion class
group is an example of a one dimensional WKD such that $I(D)\vartriangleleft
^{t}P$ where $P=$ "--- is a proper nonzero principal ideal" . (2) In section
4 of \cite{AZ 1991} there are studied several examples of Noetherian API
domains that are not integrally closed. The simplest of these is $%
\mathbb{Z}
\lbrack 2i]=%
\mathbb{Z}
+2i%
\mathbb{Z}
.$ Since for each $a+bi\in $ $%
\mathbb{Z}
\lbrack i]$ we have $(a+bi)^{2}=a^{2}-b^{2}+2abi\in 
\mathbb{Z}
\lbrack 2i],$ this gives the conclusion that $%
\mathbb{Z}
\lbrack 2i]$ is Noetherian and that $%
\mathbb{Z}
\lbrack 2i]$ $\subseteq 
\mathbb{Z}
\lbrack 2i]$ is a fixed root extension. Because $%
\mathbb{Z}
\lbrack i]$ is a PID, Corollary 4.13 of \cite{AZ 1991} applies to give the
conclusion that $%
\mathbb{Z}
\lbrack 2i]$ is an API domain. That $%
\mathbb{Z}
\lbrack 2i]$ is one dimensional, follows from Theorem 2.1 of \cite{AZ 1991}.
Now let $M$ be a maximal ideal of $%
\mathbb{Z}
\lbrack 2i].$ Then $M$ is finitely generated, say $M=(x_{1},x_{2},...,x_{r})$
then $(x_{1}^{n},...,x_{r}^{n})$ is principal and, using Lemma 2.3 of \cite%
{XAZ 2019}, we conclude that $M^{nr}\subseteq (x_{1}^{n},...,x_{r}^{n}).$
(3) Finally, let $K$ be a field of characteristic $p>0$ and let $L$ be a
purely inseparable field extension of $K$ such that $L^{p}\subseteq K$ and
consider $T=K+XL[X].$ According to the information gathered prior to Example %
\ref{Example UXD}, the only non-principal maximal ideal of $T$ is $%
XL[X]=(X,lX)_{v}$ where $l\in L\backslash K.$ Obviously $%
(X^{p},(lX)^{p})_{v}=X^{p}$ and an application of Lemma 2.3 of \cite{XAZ
2019} or direct computation gives $(XL[X])^{2p}\subseteq ((XL[X])^{2p})_{v}$ 
$=((X,lX)^{2p})_{v}\subseteq X^{p}.$ The above can serve also as examples
for part (b), but all fastfaktorielle rings of \cite{Stor 1967} dubbed as
almost factorial domains in \cite{Fos 1973} can serve as examples as almost
factorial domains are nothing but Krull domains with torsion divisor class
groups. For non-Krull examples for (b) recall that, according to \cite{Zaf
1985}, an integral domain $D$ is called an AGCD domain if for each pair $%
a,b\in D\backslash \{0\}$ there is a positive integer $n=n(a,b)$ such that $%
a^{n}D\cap b^{n}D$ is principal (equivalently for every nonzero finitely
generated ideal $(a_{1},...,a_{r})$ there is $n=n$ $(a_{1},...,a_{r})\geq 1$
such that $(a_{1}^{n},...,a_{r}^{n})_{v}$ is principal). Any Noetherian AGCD
domain would serve as an example for (b). Reason: take a maximal $t$-ideal $%
M,$ it's finitely generated. Say $M=$ $(a_{1},...,a_{r}),$ for some $n\geq 1$
we must have $(a_{1}^{n},...,a_{r}^{n})_{v}=dD,$ principal. But then $%
M^{nr}\subseteq (a_{1}^{n},...,a_{r}^{n})\subseteq
(a_{1}^{n},...,a_{r}^{n})_{v}=dD,$ by Lemma 2.3 of \cite{XAZ 2019}.
\end{example}

\section{Generalizing Almost Bezout domains\label{Section S4}}

Following \cite{AZ 2019}, in a way, we call a domain $D$, for $\ast =d$ or $%
t,$ a $\ast $-almost Bezout ($\ast $-AB) (resp. $\ast $-almost Prufer ($\ast 
$-AP)) domain if for every finite set $x_{1},x_{2},...,x_{n}\in D\backslash
(0),$ there is a natural number $r=r(x_{1},x_{2},...,x_{n})$ such that $%
(x_{1}^{r},...,x_{n}^{r})^{\ast }$ is principal (resp., $\ast $-invertible).
Indeed a $d$-AB ($d$-AP) domain is the usual almost Bezout (almost Prufer)
domain, as defined in \cite{AZ 1991} and a $t$-AB (resp., $t$-AP) domain is
the usual AGCD domain, as defined in \cite{Zaf 1985} (resp., an APVMD, as
defined in \cite{L 1991} and studied in \cite{Li 2012}). It may be noted
that if the natural number $r=r(x_{1},x_{2},...,x_{n})$ is $1$ for each set $%
x_{1},x_{2},...,x_{n}\in D\backslash \{0\},$ we get the usual Bezout
(Prufer) domain for $\ast =d$ and the usual GCD domain (resp., PVMD) for $%
\ast =t.$

Twisting the definition a little we have the following lemma.

\begin{lemma}
\label{Lemma ZA}(1) A domain $D$ is $\ast $-AB ( resp., $\ast $-AP) if and
only if for every nonzero $\ast $-prime $P$ and for every set $%
x_{1},x_{2},...,x_{n}\in P\backslash (0)$ there is a natural number $%
r=r(x_{1},x_{2},...,x_{n})$ such that the ideal $%
(x_{1}^{r},x_{2}^{r},...,x_{n}^{r})^{\ast }$ is principal ($\ast $%
-invertible)$,$ (2) A nonzero prime ideal of a $\ast $-AB ($\ast $-AP)
domain $D$ is a $\ast $-ideal if and only if for every set $%
x_{1},x_{2},...,x_{n}\in P\backslash (0)$ there is a natural number $%
r=r(x_{1},x_{2},...,x_{n})$ such that the ideal $%
(x_{1}^{r},x_{2}^{r},...,x_{n}^{r})^{\ast }$ is contained in $P.$
\end{lemma}

\begin{proof}
(1) We show that the condition implies that $D$ is a $\ast $-AB ($\ast $-AP)
domain. Let $x_{1},x_{2},...,x_{n}\in D\backslash (0).$ If $%
(x_{1},x_{2},...,x_{n})$ is not contained in any $\ast $-ideal $P,$ then $%
(x_{1},x_{2},...,x_{n})^{\ast }=D$ and so $(x_{1},x_{2},...,x_{n})^{\ast }$
is principal for $r=1.$ Now if $0\neq (x_{1},x_{2},...,x_{n})\subseteq P$
then by the condition there is a natural number $r=r(x_{1},x_{2},...,x_{n})$
such that $(x_{1}^{r},...,x_{n}^{r})^{\ast }$ is principal (resp., $\ast $%
-invertible). The converse is obvious. (2) Let $P$ be a $\ast $-ideal in a $%
\ast $-AB ($\ast $-AP) domain $D.$ Then, obviously, for every set $%
x_{1},x_{2},...,x_{n}\in P\backslash (0)$ there is a natural number $%
r=r(x_{1},x_{2},...,x_{n})$ such that the ideal $%
(x_{1}^{r},x_{2}^{r},...,x_{n}^{r})^{\ast }$ is contained in $P.$ Conversely
suppose that $P$ is a prime ideal of a $\ast $-AB ($\ast $-AP) domain $D$
such that for every set $x_{1},x_{2},...,x_{n}\in P\backslash (0)$ there is
a natural number $r=r(x_{1},x_{2},...,x_{n})$ such that the ideal $%
(x_{1}^{r},x_{2}^{r},...,x_{n}^{r})^{\ast }$ is contained in $P.$ Since $%
(x_{1},x_{2},...,x_{n})^{nr}\subseteq (x_{1}^{r},x_{2}^{r},...,x_{n}^{r})$
by Lemma 2.3 of \cite{XAZ 2019} we have $((x_{1},x_{2},...,x_{n})^{nr})^{%
\ast }$ $=(((x_{1},x_{2},...,x_{n})^{\ast })^{nr})^{\ast }\subseteq
(x_{1}^{r},x_{2}^{r},...,x_{n}^{r})^{\ast }.$ Since $P$ is a prime ideal and
since $((x_{1},x_{2},...,x_{n})^{\ast })^{nr}\subseteq P$ we conclude that $%
(x_{1},x_{2},...,x_{n})^{\ast }\subseteq P.$
\end{proof}

\begin{definition}
\label{Definition ZB} We call a nonzero finitely generated ideal $%
(x_{1},x_{2},...,x_{n}),$ power divisible by a $\ast $-invertible $\ast $%
-ideal $J$ if for some natural number $s$ we have $%
(x_{1},x_{2},...,x_{n})^{s}\subseteq J$ and restricted power divisible if $J$
is in each prime $\ast $-ideal $P$ that contains $(x_{1},x_{2},...,x_{n}).$
Call $D$ a $\ast $-sub-almost Bezout ($\ast $-SAB) ($\ast $-sub-almost
Prufer($\ast $-SAP)) domain if for every set $x_{1},x_{2},...,x_{n}\in
D\backslash (0)$, there is a natural number $s=s(x_{1},x_{2},...,x_{n})$
such that the ideal $(x_{1},x_{2},...,x_{n})^{s}$ is contained in a
principal ideal ($\ast $-invertible $\ast $-ideal) $J$ such that $J$ belongs
to each prime $\ast $-ideal that contains $(x_{1},x_{2},...,x_{n}).$ Thus $D$
is a $\ast $-SAB ($\ast $SAP) domain if every proper finite type $\ast $%
-ideal of $D$ is strictly power divisible by at least one principal ideal ($%
\ast $-invertible $\ast $-ideal). Finally let $\Gamma $ be a pre-assigned
set of proper nonzero principal ( resp., $\ast $-invertible $\ast $-) ideals
of $D.$ Call $D$ an infra $\ast $-AB$\Gamma $ (resp., infra $\ast $-AP$%
\Gamma )$ domain if for each finitely generated ideal ideal $I$ of $D$ there
is a positive integer $n=n(I)$ and a $\gamma \in \Gamma $such that $%
I^{n}\subseteq \gamma .$
\end{definition}

The following lemma indicates the relationship between some of these
concepts.

\begin{lemma}
\label{Lemma ZC} (1) A Bezout domain is a GCD domain; (2) A GCD domain is a
PVMD; (3) A Prufer domain is a PVMD, (4) A Bezout domain is an AB domain (5)
An AB domain is an AGCD domain (6) An AGCD domain is an APVMD (7) A Bezout
domain is a Prufer domain (8) A Prufer domain is an AP domain, (9) An AB
domain is an AP domain (10) A $\ast $-SAB domain is a $\ast $-SAP domain
(11) An AB domain is a an SAB domain (12) An AP domain is an SAP domain,
(13) an AGCD ($t$-AB) domain is an SAGCD ($t$-SAB) domain and (14) An APVMD
is a SAPVMD ($t$-SAP), (15) a $\ast $-SAB (resp., $\ast $-SAP) domain is a $%
\ast $-AB$\Gamma $ (resp., $\ast $-AP$\Gamma )$ domain.
\end{lemma}

\begin{proof}
While the proofs of most of the above statements are well known, we briefly
go through each just to remind the reader of the definitions. (Of course a
reader conversant with the notions can skip the proofs.)

(1) Every finitely generated (nonzero) ideal is principal implies that for
every finitely generated nonzero ideal $I$ we have $I_{v}$ principal.

(2) $I_{v}$ is principal for every finitely generated non zero $I$ implies $%
I_{v}$ is $t$-invertible for every finitely generated non zero $I$ ( it's a
case of principal is $t$-invertible).

(3) $I$ is invertible for every finitely generated non zero $I$ implies $I$
is $t$-invertible for every finitely generated non zero $I.$ (It's a case of
invertible is $t$-invertible.)

(4) For all $x_{1},x_{2},...,x_{n}\in D\backslash \{0\}$ $%
(x_{1},x_{2},...,x_{n})$ is principal implies there is $%
r=r(x_{1},x_{2},...,x_{n})=1$ such that $(x_{1}^{r},x_{2}^{r},...,x_{n}^{r})$
is principal.

(5) For all $x_{1},x_{2},...,x_{n}\in D\backslash (0)$ there is a natural
number $r=r(x_{1},x_{2},...,x_{n})$ such that the ideal $%
(x_{1}^{r},x_{2}^{r},...,x_{n}^{r})$ is principal implies for all $%
x_{1},x_{2},...,x_{n}\in D\backslash (0)$ there is a natural number $%
r=r(x_{1},x_{2},...,x_{n})$ such that the ideal $%
(x_{1}^{r},x_{2}^{r},...,x_{n}^{r})_{v}$ is principal. (Every principal
ideal is a $v$-ideal.)

(6) For all $x_{1},x_{2},...,x_{n}\in D\backslash (0)$ there is a natural
number $r=r(x_{1},x_{2},...,x_{n})$ such that the ideal $%
(x_{1}^{r},x_{2}^{r},...,x_{n}^{r})_{v}$ is principal implies that for all $%
x_{1},x_{2},...,x_{n}\in D\backslash (0)$ there is a natural number $%
r=r(x_{1},x_{2},...,x_{n})$ such that the ideal $%
(x_{1}^{r},x_{2}^{r},...,x_{n}^{r})_{v}$ is $t$-invertible.

(7) For all $x_{1},x_{2},...,x_{n}\in D\backslash (0)$, $%
(x_{1},x_{2},...,x_{n})$ principal implies for all $x_{1},x_{2},...,x_{n}\in
D\backslash (0)$, $(x_{1},x_{2},...,x_{n})$ is invertible.

(8) For all $x_{1},x_{2},...,x_{n}\in D\backslash (0)$ such that $%
(x_{1},x_{2},...,x_{n})$ is invertible implies that for all $%
x_{1},x_{2},...,x_{n}\in D\backslash (0)$ there is a natural number $%
r=r(x_{1},x_{2},...,x_{n})=1$ such that the ideal $%
(x_{1}^{r},x_{2}^{r},...,x_{n}^{r})$ is $t$-invertible.

(9) for all $x_{1},x_{2},...,x_{n}\in D\backslash (0)$ there is a natural
number $r=r(x_{1},x_{2},...,x_{n})$ such that the ideal $%
(x_{1}^{r},x_{2}^{r},...,x_{n}^{r})$ is principal implies that for all $%
x_{1},x_{2},...,x_{n}\in D\backslash (0)$ there is a natural number $%
r=r(x_{1},x_{2},...,x_{n})$ such that the ideal $%
(x_{1}^{r},x_{2}^{r},...,x_{n}^{r})$ is invertible.

(10) follows because a principal ideal is a $\ast $-invertible $\ast $-ideal.

(11) for each nonzero prime ideal $P$ of $D$ and for all $%
x_{1},x_{2},...,x_{n}\in P\backslash (0)$ the existence of $%
r=r(x_{1},x_{2},...,x_{n})$ such that $(x_{1}^{r},x_{2}^{r},...,x_{n}^{r})$
is a principal ideal contained in $P$ implies that for each nonzero prime
ideal $P$ of $D$ and for all $x_{1},x_{2},...,x_{n}\in P\backslash (0)$
there exists $s=s(x_{1},x_{2},...,x_{n})=rn$ such that $%
(x_{1},x_{2},...,x_{n})^{s}$ is contained in a principal ideal contained in $%
P,$ (indeed as $(x_{1},x_{2},...,x_{n})^{rn}\subseteq
(x_{1}^{r},x_{2}^{r},...,x_{n}^{r}),$ by Lemma 2.3 \cite{XAZ 2019}).

(12) for each nonzero prime ideal $P$ of $D$ and for all $%
x_{1},x_{2},...,x_{n}\in P\backslash (0)$ the existence of $%
r=r(x_{1},x_{2},...,x_{n})$ such that $(x_{1}^{r},x_{2}^{r},...,x_{n}^{r})$
is an invertible ideal contained in $P$ implies that for each nonzero prime
ideal $P$ of $D$ and for all $x_{1},x_{2},...,x_{n}\in P\backslash (0)$
there exists $s=s(x_{1},x_{2},...,x_{n})=rn$ such that $%
(x_{1},x_{2},...,x_{n})^{s}$ is contained in an invertible ideal $J$
contained in $P$ (and hence in every prime $\ast $-ideal $Q$ containing $%
(x_{1},x_{2},...,x_{n}).$

(13) for each nonzero prime ideal $P$ of $D$ and for all $%
x_{1},x_{2},...,x_{n}\in P\backslash (0)$ the existence of $%
r=r(x_{1},x_{2},...,x_{n})$ such that $%
(x_{1}^{r},x_{2}^{r},...,x_{n}^{r})_{t}$ is a principal ideal contained in $%
P $ implies that for each nonzero prime ideal $P$ of $D$ and for all $%
x_{1},x_{2},...,x_{n}\in P\backslash (0)$ there exists $%
s=s(x_{1},x_{2},...,x_{n})=rn$ such that $(x_{1},x_{2},...,x_{n})^{s}$ is
contained in a principal ideal $J$ contained in $P.$

(14) for each nonzero prime ideal $P$ of $D$ and for all $%
x_{1},x_{2},...,x_{n}\in P\backslash (0)$ the existence of $%
r=r(x_{1},x_{2},...,x_{n})$ such that $(x_{1}^{r},x_{2}^{r},...,x_{n}^{r})$
is a $t$-invertible ideal contained in $P$ implies that for each nonzero
prime ideal $P$ of $D$ and for all $x_{1},x_{2},...,x_{n}\in P\backslash (0)$
there exists $s=s(x_{1},x_{2},...,x_{n})=rn$ such that $%
(x_{1},x_{2},...,x_{n})^{s}$ is contained in a $t$-invertible $t$-ideal $J$
contained in $P.$

(15) Direct. Set $\Gamma =\{\gamma |\gamma $ is a principal ideal of $D\}$
for a $\ast $-SAB domain $D.$ Likewise, for a $\ast $-SAP domain set $\Gamma
=\{\gamma |\gamma $ is a $\ast $-ivertible $\ast $-ideal of $D\}.$
\end{proof}

Examples.

(1) Every $(\ast $-AB (resp., $\ast $-AP) domain is a $\ast $-SAB (resp., $%
\ast $-SAP) domain.

(2) The $D+XK[X]$ construction from a $\ast $-SAB (resp., $\ast $-SAP)
domain delivers a $\ast $-SAB (resp., $\ast $-SAP) domain. More precisely we
have the following result.

\begin{theorem}
\label{Theorem ZC1} Let $K$ be the quotient field of a domain $D$ and let $X$
be an indeterminate over $K.$ Then $D$ is a $\ast $-SAB (resp., $\ast $-SAP)
domain if and only if $R=D+XK[X]$ is.
\end{theorem}

\begin{proof}
Let $F=(f_{1},f_{2},...,f_{n}).$ Then according to \cite[Proposition 4.12]%
{CMZ 1978}, $F=f(X)(J+XK[X])$ where $f(X)\in R$ and $J$ is a finitely
generated ideal of $D.$ If $f(0)=0$ we conclude that $f(X)=g(X)(kX^{r})$
where $g(0)=1$, $k\in K\backslash \{0\}$ and $r>0.$ Then $%
F^{2}=g(X)^{2}(k^{2}X^{2r})(J+XK[X])^{2}\subseteq g(X)X(J+XK[X]).$ Obviously 
$F$ is contained in the same prime $t$-ideals as $g(X)X(J+XK[X])$ is. Next
as $D$ is a $\ast $-SAB (resp., $\ast $-SAP) domain and as $J$ is a finitely
generated ideal, for some natural number $r,$ $J^{r}\subseteq H$ where $H$
is a principal ideal (resp., $\ast $-invertble $\ast $-ideal) and $H$ is
contained in the same prime $\ast $-ideals as $J.$ But then $H+XK[X]$ is a
principal ideal (resp., $\ast $-invertble $\ast $-ideal) that is contained
in the same prime $\ast $-ideals as $J+XK[X]$ is. Thus $F^{2r}\subseteq
g(X)^{r}X^{r}(J+XK[X])^{r}\subseteq g(X)^{r}X^{r}(H+XK[X])$ which is a
principal ideal (resp., $\ast $-invertble $\ast $-ideal) that is contained
in the same prime $\ast $-ideals as $F.$ If $f(X)\neq 0,$ we can take $%
f(0)=1 $ and so $f(X)=1$ or $f(X)=p_{1}^{r_{1}}...p_{s}^{r_{s}}$ where $%
p_{i}R$ is a a height one maximal ideal, for each $i.$ In both cases $%
F^{n}=f((X)^{n}(J^{n}+XK[X])$ for all natural numbers $n$. Next as $J$ is a
finitely generated ideal of $D,$ and $D$ is a $\ast $-SAB ($\ast $SAP)
domain, we have a natural number $r$ and a principal ideal (resp., $\ast $%
-invertible $\ast $-ideal) $H$ such that $H$ is contained in each prime $%
\ast $-ideal that contains $J^{r}.$ But then, as above, $H+XK[X]$ is a
principal ideal (resp., $\ast $-invertible $\ast $-ideal) containing $%
J^{n}+XK[X]$ and contained in the same prime $\ast $-ideals $P+XK[X],$ as $%
J+K[X]$ is. Consequently. $F^{r}=f((X)^{r}(J^{r}+XK[X])\subseteq
f((X)^{r}(H+XK[X]).$

Conversely suppose that $D+XK[X]$ is a $\ast $SAB (resp., $\ast $SAP)
domain. Then in particular for every set $x_{1},x_{2},...,x_{n}\in
D\backslash (0),$ there is a natural number $r$ and a principal ideal
(resp., $\ast $-invertible $\ast $-ideal) $J$ of $R$ such that $%
((x_{1},x_{2},...,x_{n})R)^{r}\subseteq J$ where $J$ is contained in each of
the prime $\ast $-ideals $%
(x_{1},x_{2},...,x_{n})R=(x_{1},x_{2},...,x_{n})+XK[X]$ is contained in.
Obviously, as $((x_{1},x_{2},...,x_{n})R)^{r}\cap D\neq (0),$ $J=H+XK[X]$
where $H$ is a principal ideal (resp., $\ast $-invertible $\ast $-ideal)
because $J$ is. But then every prime $\ast $-ideal containing $J$ and $%
(x_{1},x_{2},...,x_{n})+XK[X]$ is of the form $P+XK[X]$ where $P$ contains $%
H $ and $(x_{1},x_{2},...,x_{n}).$ Also since $%
(x_{1},x_{2},...,x_{n})^{r}+XK[X]\subseteq J=H+XK[X]$ we have $%
(x_{1},x_{2},...,x_{n})^{r}\subseteq H.$ Indeed $H$ is contained in the same
prime $\ast $-ideals that contain $(x_{1},x_{2},...,x_{n})$. Since the
choice of $x_{1},x_{2},...,x_{n}\in D\backslash \{0\}$ is arbitrary we
conclude that $D$ is a $\ast $-SAB ($\ast $-SAP) domain.
\end{proof}

(3) For $L$ a field extension of $K$, the $D+XL[X]$ construction from a $%
\ast $-SAB (resp., $\ast $-SAP) domain delivers a $\ast $-SAB (resp., $\ast $%
-SAP) domain. More precisely we have the following result. (Theorem \ref%
{Theorem ZC2} can actually replace Theorem \ref{Theorem ZC1} but while
Theorem \ref{Theorem ZC1} delivers AGCD domain from AGCD domains directly,
Theorem \ref{Theorem ZC2} may need an adjustment.)

\begin{theorem}
\label{Theorem ZC2} Let $L$ be an extension of the quotient field $K$ of $D$
and let $X$ be an indeterminate over $L.$ Then $D$ is a $\ast $-SAB (resp., $%
\ast $-SAP) domain if and only if $R=D+XL[X]$ is.
\end{theorem}

\begin{proof}
Let $F=(f_{1},f_{2},...,f_{n}).$ Then according to \cite[Proposition 3]{Z
2020}, $F=f(X)(J+XK[X])$ where $f(X)\in R$ and $J$ is a finitely generated $%
D $-submodule of $L.$ If $f(0)=0,$ we conclude that $f(X)=g(X)(lX^{r})$
where $g(0)=1$, $l\in L\backslash \{0\}$ and $r>0.$ This gives $%
F^{2}=g(X)^{2}(l^{2}X^{2r})(J+XL[X])^{2}\subseteq g(X)X(J+XL[X]).$ Obviously 
$F$ is contained in the same prime $t$-ideals as $g(X)X(J+XL[X])$ is. Two
cases arise here: (a) $J$ is an ideal of $D$ and (b) $J$ is a $D$-submodule
of $L$ such that $J$ is not contained in $D.$ In case (b), $XL[X]$ is the
only prime $\ast $-ideal containing $X(J+XL[X]].$ Thus $g(X)X(J+XL[X])$ is
contained in the same prime $t$-ideals that $g(X)X$ is contained in. But $%
F^{4}\subseteq g(X)XR,$ which is a principal ideal and hence a $\ast $%
-invertible $\ast $-ideal. In case (a) as $D$ is a $\ast $-SAB (resp., $\ast 
$-SAP) domain and as $J$ is a finitely generated ideal of $D$, for some
natural number $r,$ $J^{r}\subseteq H$ where $H$ is a principal ideal
(resp., $\ast $-invertble $\ast $-ideal) and $H$ is contained in the same
prime $\ast $-ideals as $J.$ But then $H+XL[X]$ is a principal ideal (resp., 
$\ast $-invertble $\ast $-ideal) that is contained in the same prime $\ast $%
-ideals as $J+XL[X]$ is. Thus $F^{2r}\subseteq
g(X)^{r}X^{r}(J+XL[X])^{r}\subseteq g(X)^{r}X^{r}(H+XL[X])$ which is a
principal ideal (resp., $\ast $-invertble $\ast $-ideal) that is contained
in the same prime $\ast $-ideals as $F.$ If $f(X)\neq 0,$ we can take $%
f(0)=1 $ and so $f(X)=1$ or $f(X)=p_{1}^{r_{1}}...p_{s}^{r_{s}}$ where $%
p_{i}R$ is a a height one maximal ideal, for each $i$ and $J$ is a finitely
generated ideal of $D.$ In both cases $F^{n}=f((X)^{n}(J^{n}+XL[X])$ for all
natural numbers $n$. Next as $J$ is a finitely generated ideal of $D,$ and $%
D $ is a $\ast $-SAB ($\ast $-SAP) domain, we have a natural number $r$ and
a principal ideal (resp., $\ast $-invertible $\ast $-ideal) $H$ such that $H$
is contained in each prime $\ast $-ideal that contains $J^{r}.$ But then, as
above, $H+XL[X]$ is a principal ideal (resp., $\ast $-invertible $\ast $%
-ideal) containing $J^{n}+XL[X]$ and contained in the same prime $\ast $%
-ideals $P+XL[X],$ as $J+L[X]$. Thus $F^{r}=f((X)^{r}(J^{r}+XL[X])\subseteq
f((X)^{r}(H+XL[X]).$

Conversely suppose that $D+XL[X]$ a $\ast $-SAB ($\ast $-SAP) domain. Then,
in particular, for every set $x_{1},x_{2},...,x_{n}\in D\backslash (0),$
there is a natural number $r$ and a principal ideal (resp., $\ast $%
-invertible $\ast $-ideal) $J$ of $R$ such that $%
((x_{1},x_{2},...,x_{n})R)^{r}\subseteq J$ where $J$ is contained in each of
the prime $\ast $-ideals $%
(x_{1},x_{2},...,x_{n})R=(x_{1},x_{2},...,x_{n})+XL[X]$ is contained in.
Obviously, as $((x_{1},x_{2},...,x_{n})R)^{r}\cap D\neq (0),$ $%
J=H+XL[X]=H(D+XL[X])$ where $H$ is a principal ideal (resp., $\ast $%
-invertible $\ast $-ideal) because $J$ is. But then every prime $\ast $%
-ideal containing $J$ and $(x_{1},x_{2},...,x_{n})+XL[X]$ is of the form $%
P+XL[X]$ where $P$ contains $H$ and $(x_{1},x_{2},...,x_{n}).$ Also since $%
(x_{1},x_{2},...,x_{n})^{r}+XL[X]\subseteq J=H+XL[X]$ we have $%
(x_{1},x_{2},...,x_{n})^{r}\subseteq H.$ Indeed $H$ is contained in the same
prime $\ast $-ideals that contain $(x_{1},x_{2},...,x_{n})$. Since the
choice of $x_{1},x_{2},...,x_{n}\in D\backslash \{0\}$ is arbitrary we
conclude that $D$ is a $\ast $-SAB ($\ast $-SAP) domain.
\end{proof}

\begin{remark}
\label{Remark ZC3} Since by definition a $\ast $-SAB domain is a $\ast $-SAP
domain, any result proved for a $\ast $-SAP domain will hold for a $\ast $%
-SAB domain, especially in the context of finiteness of character of these
domains. However, before we get to that , we take a somewhat more general
line.
\end{remark}

A domain $D$ is said to be of finite $\ast $-character if every nonzero
non-unit $x$ of $D$ belongs to only a finite number of maximal $\ast $%
-ideals. There have been several efforts at characterizing when a domain is
of finite $\ast $-character e.g. \cite{DZ 2010}, \cite{DLMZ 2001} etc.. But,
as we shall see below, the following is the most comprehensive treatment,
that may even take care of some kinks in earlier efforts.

Let $f(D)$ (resp., $f_{t}(D))$ be the set of proper nonzero finitely
generated ideals (proper $t$-ideals of finite type) and let $f_{\ast }%
\mathcal{(}D\mathcal{)}$ denote $f(D)$ (for $\ast =d)$ (resp., $f_{t}(D$ for 
$\ast =t)$. Also let $P$ be a property that defines a non-empty subset $%
\Gamma $ of proper ideals of $f(D),$ (resp., proper members of $f_{t}\left(
D\right) ).$ We say that $f_{\ast }\mathcal{(}D\mathcal{)}$ meets $P$ $($%
denoted $f_{\ast }\mathcal{(}D\mathcal{)}\vartriangleleft P)$ if for all $%
I\in f_{\ast }\mathcal{(}D\mathcal{)}$ we have $I\subseteq \gamma $ for some 
$\gamma \in \Gamma .$ We also say, as we have done at the start, that $%
f_{\ast }\mathcal{(}D\mathcal{)}$ meets $P$ with a twist (denoted $f_{\ast }%
\mathcal{(}D\mathcal{)}\vartriangleleft ^{t}P)$ if for all $I\in f_{\ast }%
\mathcal{(}D\mathcal{)}$ we have $I^{n}\subseteq \gamma $ for some positive
integer $n=n(I)$ and for some $\gamma \in \Gamma .$ Since $f_{\ast }\mathcal{%
(}D\mathcal{)}\vartriangleleft ^{t}P$ reduces to $f_{\ast }\mathcal{(}D%
\mathcal{)}\vartriangleleft P$ when we require $n=n(I)=1,$ we shall mainly
deal with $f_{\ast }\mathcal{(}D\mathcal{)}\vartriangleleft ^{t}P.$ (It may
be noted that domains $D$ satisfying $f_{\ast }\mathcal{(}D\mathcal{)}%
\vartriangleleft ^{t}P$ are nothing but infra $\ast $-AB$\Gamma $ (resp.,
infra $\ast $-AB$\Gamma )$ domains which, as we shall later see, are
different from the $\ast $-SAB ($\ast $-SAP) domains.) In this regard we
need to make some preparation. Let $I$ be a proper $\ast $-ideal of finite
type of $D.$ By the Span of $I$ (denoted $Span(I))$ we mean the set of
factors, from $\Gamma ,$ of all positive integral powers of $I.$ We say that 
$\ast $-ideals of finite type of $D$ satisfy Conrad's twisted condition $%
(F^{t})$ if for each proper finite type $\ast $-ideal $I,$ $Span(I)$ does
not contain an infinite sequence of mutually $\ast $-comaximal members.
Let's also call a finite type $\ast $-ideal $I$ $\ast $-homogeneous if $%
Span(I)$ contains no two $\ast $-comaximal members of $\Gamma $. We note
that if $J$ contains a power of $I$ then $Span(J)\subseteq Span(I).$ Also if 
$I$ and $J$ are $\ast $-comaximal then $Span(I)\cap Span(J)=\phi .$

\begin{lemma}
\label{Lemma ZC4} Let $D$ be a domain such that $f_{\ast }\mathcal{(}D%
\mathcal{)}\vartriangleleft ^{t}P.$ Then (1) $I$ is a $\ast $-homogeneous
ideal of $D$ if and only if $I$ is contained in a unique maximal $\ast $%
-ideal $M$, (2) $M=\{x\in D|(x,I)^{\ast }\neq D\}=M(I)$, (3) $I,J$ are $\ast 
$-comaximal if and only if $Span(I)\cap Span(J)=\phi $ and (4) Let $I$ and $%
J $ be $\ast $-homogeneous. Then $(I,J)^{\ast }=D$ or $M(I)=M(J).$
\end{lemma}

\begin{proof}
Suppose that $I$ is a $\ast $-homogeneous ideal and that $I$ is contained in
two maximal $\ast $-ideals $M_{1},M_{2}.$ Since $M_{i}$ are distinct we can
assume that there is $x\in M_{1}\backslash M_{2}$ and so $(x,M_{2})^{\ast
}=D $ or $(x,x_{1},...,x_{r})^{\ast }=D,$ where $x_{1},...,x_{r}\in M_{2}.$
This gives us $(I,x)^{\ast }\subseteq M_{1}$ and $(I,x_{1},...,x_{r})^{\ast
}\subseteq M_{2}.$ Now, by definition, there is a natural number $m$ such
that $I^{m}\subseteq (I,x)^{m}\subseteq \gamma _{1}\in \Gamma $ and another
natural number $n$ such that $I^{n}\subseteq
(I,x_{1},...,x_{r})^{n}\subseteq \gamma _{2}\in \Gamma .$ But then, $Span(I)$
contains two $\ast $-comaximal ideals, $\gamma _{1},\gamma _{2},$ a
contradiction. Conversely suppose that $I$ is contained in a unique maximal
ideal $M$. Then every factor of every power of $I$ is contained in $M$ and
so no pair of those factors can be $\ast $-comaximal. For part (2) let $%
I\subseteq M$ and let, for $x\in D,$ $(x,I)^{\ast }\neq D.$ This requires
that $(x,I)^{\ast }$ is contained in some maximal $\ast $-ideal, but $M$ is
the only maximal $\ast $-ideal that can contain $I,$ whence $x\in M.$ For
(3) we prove the contrapositive i.e. $Span(I)\cap Span(J)\neq \phi
\Leftrightarrow (I,J)^{\ast }\neq D.$ Now if $Span(I)\cap Span(J)\neq \phi ,$
then there is $H\in \Gamma $ such that $I^{n},J^{m}\subseteq H.$ But as $H$
is a proper ($\ast $-) ideal, $(I,J)^{\ast }\neq D.$ Conversely suppose $%
(I,J)^{\ast }\neq D$. Then, since we are working in a domain $D$ with $%
f_{\ast }\mathcal{(}D\mathcal{)}\vartriangleleft ^{t}P,$ there is a natural
number $n$ and a $\gamma \in \Gamma $ such that $(I,J)^{n}\subseteq \gamma .$
But this gives $I^{n},J^{n}\subseteq \gamma ,$ which means $Span(I)\cap
Span(J)\neq \phi .$ Finally for (4) note that if $(I,J)^{\ast }\neq D$ then,
by definition, $(I,J)^{n}\subseteq \gamma ,$ for some natural number $n$ and
for some $\gamma \in \Gamma .$ This gives $\gamma \in $ $Span(I)\cap
Span(J). $ Forcing $\gamma $ to be $\ast $-homogeneous and hence contained
in a unique maximal $\ast $-ideal, $M(\gamma )$. But then it is easy to see
that $M(I)=M(\gamma )=M(J).$
\end{proof}

\begin{theorem}
\label{Theorem ZC5} Given that $D$ is such that $f_{\ast }\mathcal{(}D%
\mathcal{)}\vartriangleleft ^{t}P$ and $D$ satisfies Conrad's twisted
condition $(F^{t}).$ For each $I\in f_{\ast }\mathcal{(}D\mathcal{)},$ $I$
is contained in at most a finite number of maximal $\ast $-ideals.
Conversely if $D$ is of finite $\ast $-character then $D$ satisfies $(F^{t})$
\end{theorem}

\begin{proof}
We first show that for every $I\in $ $f_{\ast }\mathcal{(}D\mathcal{)}$,
there is a positive integer $n$ such that $I^{n}$ is contained in some $\ast 
$-homogeneous ideal $J\in \Gamma .$ That is for every $I\in f_{\ast }%
\mathcal{(}D\mathcal{)},$ $Span(I)$ contains at least one $\ast $%
-homogeneous member of $\Gamma $. Suppose that for some $I\in $ $f_{\ast }%
\mathcal{(}D\mathcal{)}$, $Span(I)$ does not contain any $\ast $-homogeneous
ideal. Then $Span(I)$ contains at least two $\ast $-comaximal members.
Noting that each of the two $J_{1},J_{2}$ has at least two $\ast $-comaximal
members of $\Gamma $ in its span, we conclude that there are at least four
mutually $\ast $-comaximal members of $\Gamma $ say $J_{i1},J_{i2}\in
Span(J_{i}),$ in view of the fact that $\cap Span(J_{i})=\phi .$ Also since $%
Span(J_{i})\subseteq Span(I)$ we conclude that at stage $2$ $Span(I)$
contains at least four $=2^{2}$ mutually $\ast $-comaximal members of $%
\Gamma $. Since the spans of the resulting four mutually $\ast $-comaximal
members are mutually disjoint, contained in $Span(I)$ are mutually disjoint
and since, by the condition, span of each of these four contains at least
two mutually $\ast $-comaximal members of $\Gamma ,$ we conclude that at
stage $3$ $Span(I)$ has at least $2^{3}$ mutually $\ast $-comaximal members
of $\Gamma $. Proceeding thus, we may assume that at stage $n-1$ there are
at least $2^{n-1}$ mutually $\ast $-comaximal members of $\Gamma $, $%
K_{1},K_{2}....,K_{2^{n-1}}$ in $Span(I).$ Again, noting that $Span(K_{i})$
are mutually disjoint and because $Span(I)$ contains no $\ast $-homogeneous
members of $\Gamma $, each of $Span(K_{i})$ contains at least two $\ast $%
-comaximal $\ast $-invertible $\ast $-ideals. But then, at stage $n,~Span(I)$
contains at least $2^{n}$ mutually $\ast $-comaximal members. Now because of
the assumption that $Span(I)$ contains no $\ast $-homogeneous members of $%
\Gamma ,$ this process is never ending and forces $Span(I)$ to have
infinitely many mutually $\ast $-comaximal members from $\Gamma $. But this
is contrary to the twisted $(F^{t})$ condition of Conrad's. Whence, for some 
$L$ in $Span(I),$ $Span(L)$ contains a $\ast $-homogeneous ideal $J.$ But as 
$I^{m}\subseteq L$ for some $m$ and $L^{n}\subseteq J$ for some $n$ we
conclude that $I^{mn}\subseteq J.$

Next, let $H(I)$ be the set of all $\ast $-homogeneous ideals contained in $%
Span(I)$ and note that "is non $\ast $-comaximal with" is an equivalence
relation on $H(I)$ (this equivalence relation splits $H(I)$ into equivalence
classes) and that members in distinct classes are mutually $\ast $%
-comaximal. Now pick one $\ast $-homogeneous ideal from each class to get a
set of mutually $\ast $-comaximal $\ast $-homogeneous ideals. By Conrad's
twisted condition there must be a finite number $l$ of $\ast $-homogeneous
mutually $\ast $-comaximal members in $Span(I),$ say $h_{1},h_{2},...,h_{l}$
and by construction this number is exact.

Let $M(h_{i})$ be the maximal $\ast $-ideal containing $h_{i},$ for $%
i=1,...,l.$ Claim that $\{M(h_{i})\}$ are the only maximal $\ast $-ideals
that contain $I$. For if not and if $M$ is another maximal $\ast $-ideal
containing $I,$ then $M$ contains $(I,x)$ where $x\in M\backslash \cup
\{M(h_{i})\}.$ But then, as we have shown above, there is a $\ast $%
-homogeneous ideal $\gamma \in \Gamma $ and a natural number $n$ such that $%
(I,x)^{n}\subseteq \gamma .$ This forces $\gamma \in Span(I).$ But then $%
\gamma $ has to be non- $\ast $-comaximal to one of $h_{i}$ and hence in $%
M(h_{i}).$ But that is impossible because $x\notin M(h_{i}).$ Consequently, $%
I$ is contained exactly in $M(h_{1}),...,M(h_{l}).$

Conversely suppose that $D$ is of finite $\ast $-character, then we show
that for any $\ast $-ideal $I$ of finite type $Span(I)$ contains at most a
finite number of mutually $\ast $-comaximal $\ast $-ideals from $\Gamma .$
Suppose on the contrary that $Span(I)$ contains infinitely many mutually $%
\ast $-comaximal members of $\Gamma $ and let $\{\gamma _{i}\}$ be a list of
the mutually $\ast $-comaximal elements of $\Gamma $ contained in $Span(I).$
If $M_{\gamma _{i}}$ is a maximal $\ast $-ideal containing $\gamma _{i}$
then since $I^{\alpha _{i}}\subseteq \gamma _{i}$ $\subseteq M_{\gamma _{i}}$
we have $I\subseteq M_{\gamma _{i}}.$ Now as $M_{\gamma _{i}}$ cannot
contain two $\ast $-comaximal ideals, all the $M_{\gamma _{i}}$ are
distinct. Thus we end up with an infinite set of distinct maximal $\ast $%
-ideals containing $I.$ But this contradicts the assumption that $D$ is of
finite $\ast $-character. Hence $Span(I)$ contains only a finite number of
mutually $\ast $-comaximal ideals from $\Gamma .$
\end{proof}

\bigskip

\begin{corollary}
\label{Corollary ZC6} Let $D$ be a $\ast $-SAB domain. Then $D$ is of finite 
$\ast $-character if and only if every power of a proper $\ast $-ideal of
finite type is contained in an at most a finite number of mutually $\ast $%
-comaximal members of $\Gamma .$
\end{corollary}

\begin{proof}
The proof follows from the fact that a $\ast $-SAB domain domain is a
special case of infra $\ast $-AB$\Gamma $ domains.
\end{proof}

Our aim now is to record the consequences of the following statement: An
integral domain $D$ is of finite $\ast $-character if and only if (a) every $%
\ast $-locally finitely generated $\ast $-ideal of $D$ is of finite type if,
and only if, (b) Every nonzero $\ast $-ideal $A$ of finite type of $D$ is
power divisible by at most a finite number of mutually $\ast $-comaximal
members of $\Gamma .$ Here an ideal $A$ is $\ast $-locally finitely
generated if $AD_{M}$ is finitely generated for each maximal $\ast $-ideal $%
M $ of $D.$

\begin{lemma}
\label{Lemma ZC7} Let $M$ be a maximal $t$-ideal of an integral domain $D$
and suppose that $M$ has the property that every nonzero finitely generated
proper ideal $A$ of $D$ is strictly power divisible by a $t$-invertible $t$%
-ideal of $D$ contained in $M$. Then $MD_{M}$ is a $t$-ideal.
\end{lemma}

\begin{proof}
Suppose that $MD_{M}$ is not a $t$-ideal. Then for a finitely generated $%
F\subseteq MD_{M}$ we have $F=fD_{M}\subseteq MD_{M},$ where $f$ is a
finitely generated nonzero ideal contained in $M,$ such that, for $v_{1}=v$%
-operation in $D_{M},$ $%
F_{v_{1}}=(fD_{M})_{v_{1}}=(f_{v}D_{M})_{v_{1}}=D_{M}=$ $%
(F^{m})_{v_{1}}=(f^{m}D_{M})_{v_{1}}=(f_{v}^{m}D_{M})_{v_{1}}=D_{M},$ for
all natural numbers $m.$ On the other hand, if $f=(x_{1},...,x_{n}),$ then $%
f^{r}\subseteq I$ where $I$ is a $t$-invertible $t$-ideal contained in $M.$
But then $(F^{s})=(f^{s}D_{M})\subseteq ID_{M}\subseteq MD_{M}$, which
forces $(F^{s})_{v_{1}}=(f^{s}D_{M})_{v_{1}}=(f_{v}^{m}D_{M})_{v_{1}}%
\subseteq MD_{M},$ a contradiction.
\end{proof}

Noting that if $D$ is an SAP (i.e., a $d$-SAP) domain every prime ideal is a 
$t$-ideal the following theorem can be stated for both $d$- and $t$-SAP
domains.

\bigskip

\begin{theorem}
\label{Theorem ZC8} Let $D$ be a $\ast $-SAP domain. If every nonzero $\ast $%
-ideal of $D$ that is $\ast $-locally finitely generated is of finite type,
every proper nonzero principal ideal of $D$ is power divisible by at most a
finite number of mutually $\ast $-comaximal ideals of finite type.
\end{theorem}

\begin{proof}
Suppose that $A$ is power divisible by an infinite set $\{I_{\alpha }\}$ of
mutually $\ast $-comaximal $\ast $-invertible $\ast $-ideals. Choose an $%
x\in A\backslash \{0\}.$ Let $n_{\alpha }$ be, the smallest such that $%
x^{n_{\alpha }}D\subseteq I_{\alpha }.$ Two cases are possible: (a) there is 
$n\geq 1$ such that $x^{n}D\subseteq I_{\alpha _{i}},$ for infinitely many $%
\alpha _{i}$ and (b) there is no fixed $n$ such that $x^{n}D\subseteq
I_{\alpha _{i}}$ for infinitely many $\alpha _{i}.$ In case (a) using \cite[%
Lemmas 1,2 and Proposition 4]{Zaf 2010} set $B=\sum_{j=1}^{\infty }x^{n}(\Pi
_{i=1}^{i}I_{\alpha _{i}})^{-1}$ to get a $\ast $-locally principal ideal
that is a $\ast $-ideal, being an ascending union $\dbigcup x^{n}(\Pi
_{i=1}^{i}I_{\alpha _{i}})^{-1}$ of $\ast $-ideals. Yet $B$ is not a $\ast $%
-ideal of finite type, as established in Proposition 4 of \cite{Zaf 2010}.
For case (b), note that there is no single power $n$ of $x$ such that an
infinite number of $I_{\alpha }$ contains $x^{n},$ for then (a) applies. So
for a fixed power $x^{n}$ there are only finitely many $I_{\alpha }$ that
divide $x^{n}.$ Also as $I_{\alpha }$ are mutually $\ast $-comaximal, if $%
x^{n}\subseteq I_{\alpha _{i}}$ for $i=1,...,j$, then $x^{n}D\subseteq
(\dprod\limits_{i=1}^{j}I_{\alpha _{i}})^{\ast }.$ Now let $n_{1}$ be the
least positive integer such that $x^{n_{1}}$ is contained in at least one of
the $I_{\alpha }$ and let $\{I_{\alpha _{n_{1}1}},I_{\alpha
_{n_{1}2}},...,I_{\alpha _{n_{1}j_{1}}}\}$ be the set of all the members of $%
\{I_{\alpha }\}$ that contain $x^{n_{1}}D.$ Then, as we have seen above, $%
x^{n_{1}}D\subseteq (\dprod\limits_{i=1}^{j_{1}}I_{\alpha _{n_{1}i}})^{\ast
}=J_{1}.$ Next set $n_{2}$ as the least such integer that $%
x^{n_{2}}D\subseteq (\dprod\limits_{i=1}^{j_{1}}I_{\alpha _{n_{1}i}})^{\ast
} $ and at least one ideal from $\{I_{\alpha }\}\backslash \{I_{\alpha
_{n_{1}1}},I_{\alpha _{n_{1}2}},...,I_{\alpha _{n_{1}j_{1}}}\}.$ Let $%
\{I_{\alpha _{n_{2}1}},I_{\alpha _{n22}},...,I_{\alpha _{n_{2}j_{2}}}\}$ be
the set of all the members of $\{I_{\alpha }\}\backslash \{I_{\alpha
_{n_{1}1}},I_{\alpha _{n_{1}2}},...,I_{\alpha _{n_{1}j_{1}}}\}$ that contain 
$x^{n_{2}}D.$ Thus giving us $x^{n_{2}}D\subseteq
(\dprod\limits_{i=1}^{j_{1}}I_{\alpha
_{n_{1}i}}\dprod\limits_{i=1}^{j_{2}}I_{\alpha _{n_{2}i}})^{\ast
}=J_{1}J_{2}.$ Continuing thus and setting up $J_{k}=(\dprod%
\limits_{i=1}^{j_{k}}I_{\alpha _{n_{k}i}})^{\ast }$ we get $%
x^{n_{k}}D\subseteq J_{1}J_{2}...J_{k}$ and so $x^{n_{k}}\dprod%
\limits_{i=1}^{k}J_{i}^{-1}\subseteq D.$ It is easy to see that $%
n_{1}\,<n_{2}<...<n_{k}<n_{k+1}....$

Set $B=(\dsum\limits_{k=1}^{\infty
}x^{n_{k}}\dprod\limits_{i=1}^{k}J_{i}^{-1})_{\ast _{w}}$ (Note that as we
are only dealing with $t$- and $d$- operations, $\ast _{w}=w$ or $d.)$ Claim 
$B\neq D.$ For $B=D$ implies $D=(\dsum\limits_{r=1}^{s}x^{n_{k_{r}}}\dprod%
\limits_{i=1}^{k_{r}}J_{i}^{-1})_{\ast _{w}}$ for finite $s.$ Assume $%
n_{k_{1}}<...<n_{k_{s}}.$ Then, using the fact that $%
(A_{1,}A_{2},...,A_{r})_{\ast _{w}}=D$ if and only if $%
(A_{1}^{n_{1}},A_{2}^{n_{2}},...,A_{r}^{n_{r}})_{\ast _{w}}=D,$ for any
nonnegative integers $n_{i},$ $D=(\dsum\limits_{r=1}^{s}x^{n_{k_{s}}}(\dprod%
\limits_{i=1}^{k_{r}}J_{i}^{-1})(\dprod%
\limits_{i=1}^{k_{r}}J_{i}^{-1})^{k_{s}-k_{r}})_{\ast
_{w}}=x^{n_{k_{s}}}((\dprod%
\limits_{i=1}^{k_{r_{1}}}J_{i}^{-1})^{k_{s}-k_{1}},(\dprod%
\limits_{i=1}^{k_{r_{2}}}J_{i}^{-1})^{k_{s}-k_{2}},...,(\dprod%
\limits_{i=1}^{k_{r_{s}}}J_{i}^{-1}))_{\ast _{w}}.$ This gives $%
x^{n_{k_{s}}}=((\dprod\limits_{i=1}^{k_{r_{1}}}J_{i})^{k_{s}-k_{1}})^{\ast
}\cap ...\cap (\dprod\limits_{i=1}^{k_{r_{s}}}J_{i})^{\ast }$ forcing $%
x^{n_{k_{s}}}$ to be a $\ast $-product of powers of $J_{i},i=1,2,...,s.$ But
then higher powers of $x$ cannot be divisible by the remaining members of $%
\{I_{\alpha }\}$ which contradicts the assumption. Let's note that $B$ is
not of finite type for if that were the case we would have,say, $%
B=(\dsum\limits_{r=1}^{s}x^{n_{k_{r}}}\dprod%
\limits_{i=1}^{k_{r}}J_{i}^{-1})_{\ast _{w}}$ for finite $s.$ Assume $%
n_{k_{1}}<...<n_{k_{s}}.$ But, by construction, $B\supseteq x^{n_{k_{s+1}}}$ 
$\dprod\limits_{i=1}^{k_{s+1}}J_{i}^{-1}$ or $(\dsum%
\limits_{r=1}^{s}x^{n_{k_{r}}}\dprod\limits_{i=1}^{k_{r}}J_{i}^{-1})_{\ast
_{w}}\supseteq x^{n_{k_{s+1}}}$ $\dprod\limits_{i=1}^{k_{s+1}}J_{i}^{-1}$ or 
$x^{n_{k_{s+1}}-n_{k_{1}}}\subseteq
(\dprod\limits_{i=1}^{k_{s+1}}J_{i}(\dsum%
\limits_{r=1}^{s}x^{n_{k_{r}}-n_{k_{1}}}\dprod%
\limits_{i=1}^{k_{r}}J_{i}^{-1}))_{\ast _{w}}$ or $x^{n_{k_{s+1}}-n_{k_{1}}}%
\subseteq
((\dsum\limits_{r=1}^{s}x^{n_{k_{r}}-n_{k_{1}}}\dprod%
\limits_{i=1}^{k_{s+1}}J_{i}\dprod\limits_{i=1}^{k_{r}}J_{i}^{-1}))_{\ast
_{w}}\subseteq J_{k_{s+1}}$ which is impossible.

On the other hand $B_{\ast _{w}}$ is $\ast $-locally finitely generated, as
we see below. Note that as $MD_{M}$ is a $t$-ideal, only one of th\NEG{e} $%
J_{i}$ can be a non-unit and also as $n_{1}\,<n_{2}<...<n_{k}<n_{k+1}....,$
we conclude that $B_{\ast _{w}}D_{M}=$ $BD_{M}=$ $(\dsum\limits_{k=1}^{%
\infty }x^{n_{k}}\dprod\limits_{i=1}^{k}J_{i}^{-1})D_{M}$ $%
=\sum_{j=1}^{j=k-1}x^{n_{1}}D_{M}+x^{n_{k}}J_{k}^{-1}D_{M}+\sum_{i=1}^{%
\infty }x^{n_{k+i}}J_{k}^{-1}D_{M}=x^{n_{1}}D_{M}+x^{n_{k}}J_{k}^{-1}D_{M}.$
Since $J_{k}$ is a $\ast $-invertible $\ast $-ideal, we have $%
J_{k}^{-1}D_{M} $ principal and thus $B_{w}D_{M}=$ $%
BD_{M}=x^{n_{1}}D_{M}+x^{n_{k}}J_{k}^{-1}D_{M}$ is two generated. Thus $%
B_{w} $ is $t$-locally finitely generated, yet as we have already seen $B_{w}
$ is not of finite type, a contradiction. This contradiction establishes
that if every $w$-ideal that is $t$-locally finitely generated is of finite
type, then every nonzero non-unit is power divisible by at most a finite
number of mutually $\ast $-comaximal $\ast $-invertible $\ast $-ideals.
\end{proof}

\begin{remark}
\label{Remark ZC8A} The proof of Theorem \ref{Theorem ZC8} closely follows
the proof of Lemma 2.2 of \cite{CH 2019}, except that where they "adopt" a
procedure from \cite{Zaf 2010} without reference, I reference \cite{Zaf 2010}
and when they say, in their study of case 2: "As in case 1 let $A=(\sum
a^{n_{i}}J_{i}^{-1})_{w}",$ I first establish that the ideal $A$ is proper.
(Let's put it this way: There's no place for "let" in the middle of an
argument.
\end{remark}

\begin{corollary}
\label{Corollary ZC9} Let $D$ be a $\ast $-SAP domain then $D$ is of finite $%
\ast $-character if and only if every $\ast $-locally finitely generated $%
\ast $-ideal of $D$ is of $\ast $-finite type.
\end{corollary}

\begin{proof}
By Theorem \ref{Theorem ZC8}, every $\ast $-locally finitely generated ideal 
$I$ of $D$ being of finite type forces $D$ to satisfy Conrad's ($F^{t}$) and
by Theorem \ref{Theorem ZC5} $D$ is of finite $\ast $-character. The
converse is easy to construct.
\end{proof}

\begin{corollary}
\label{Corollary ZC10} (cf. \cite{CH 2019}) Let $D$ be a $\ast $-AP domain.
If every $\ast $-locally finitely generated $\ast _{w}$-ideal of $D$ is a $%
\ast $-ideal of finite type then $D$ is of finite $\ast $-character.
\end{corollary}

Finally let's note that a domain satisfying $f_{\ast }\mathcal{(}D\mathcal{)}%
\vartriangleleft ^{t}P$ (or an infra $\ast $-AB$\Gamma )$ domain may not
necessarily be a $\ast $-SAB or a $\ast $-SAP domain. For this note that
every pre-Schreier domain $D$ has the property that if $%
(x_{1},...,x_{r})_{v}\neq D$ then there is a non-unit $d\in D$ such that $%
(x_{1},...,x_{r})_{v}\subseteq dD$ \cite[Lemma 2.1]{Zaf 1990}. But there is
a pre-Schreier domain $D$ with a maximal $t$-ideal $M$ with $MD_{M}$ not a $%
t $-ideal (see e.g. Example 2.7 of \cite{Zaf 2020}). On the other hand a $%
\ast $-SAP and hence a $\ast $-domain is well behaved, according to Lemma %
\ref{Lemma ZC7}.

\begin{acknowledgement}
Input towards improved presentation from Dan Anderson and G.M. Bergman is
gratefully acknowledged. (Of course, flubs and mistakes are all mine.)
\end{acknowledgement}

Conflict of interest: There is no conflict of interest.

Corresponding Author: Being the only author of this work, I (Muhammad
Zafrullah) am the corresponding author.

\bigskip

\end{document}